\documentclass[a4paper,twoside]{article}  

\usepackage{amssymb,amsmath,amsthm,dsfont,amsfonts,color,latexsym}   
\usepackage{hyperref}
\usepackage{CJK}
\usepackage{mathrsfs} 

\definecolor{blue}{rgb}{0.00,0.00,1.00}
\definecolor{red}{rgb}{1.00,0.00,0.00}
\newcommand{\blue}{\color{blue}}

\renewcommand{\baselinestretch}{1.2}
\makeatletter

\newcommand{\Rmnum}[1]{\expandafter\@slowromancap\romannumeral #1@}
\makeatother
\def\bq{\begin{equation}}
	\def\eq{\end{equation}}
\def\ba{\begin{array}{ccc}}
	\def\bal{\begin{array}{lll}}
		\def\ea{\end{array}}

	\def\({\left(}\def\){\right)}
	\def\[{\left[}\def\]{\right]}
	
	\def\bq{\begin{equation}}
		\def\eq{\end{equation}}
	\def\ba{\begin{array}{ccc}}
		\def\bal{\begin{array}{lll}}
			\def\ea{\end{array}}

		\def\({\left(}\def\){\right)}
		\def\[{\left[}\def\]{\right]}
		\def\<{\langle}\def\>{\rangle}

		\def \C   {\mathbb{C}}

		\def \R   {\mathbb{R}}

		\def\eps  {\epsilon}

		\def\intt {\int^t_0}

		\def \dx    {\partial_x}

		\def \dxa   {\partial^{\alpha}_x}

		\def\bq{\begin{equation}}
			\def\eq{\end{equation}}
		\def\be{\begin{equation}}
			\def\ee{\end{equation}}
		\def\bma#1\ema{{\allowdisplaybreaks\begin{align}#1\end{align}}}
		\def\bmas#1\emas{{\allowdisplaybreaks\begin{align*}#1\end{align*}}}
		\def\bln#1\eln{{\allowdisplaybreaks\begin{aligned}#1\end{aligned}}}
		\def\nnm{\notag}
		\def\bgr#1\egr{\allowdisplaybreaks\begin{gather}#1\end{gather}}
		\def\bgrs#1\egrs{\allowdisplaybreaks\begin{gather*}#1\end{gather*}}

		\def\bq{\begin{equation}}
			\def\eq{\end{equation}}
		\def\be{\begin{equation}}
			\def\ee{\end{equation}}
		\def\bma#1\ema{{\allowdisplaybreaks\begin{align}#1\end{align}}}
		\def\bmas#1\emas{{\allowdisplaybreaks\begin{align*}#1\end{align*}}}
		\def\bln#1\eln{{\allowdisplaybreaks\begin{aligned}#1\end{aligned}}}
		\def\nnm{\notag}
		\def\bgr#1\egr{\allowdisplaybreaks\begin{gather}#1\end{gather}}
		\def\bgrs#1\egrs{\allowdisplaybreaks\begin{gather*}#1\end{gather*}}

		\theoremstyle{plain}
		\newtheorem{lem}{\bf Lemma}[section]
		\newtheorem{thm}[lem]{\textbf{Theorem}}
		\newtheorem{prop}[lem]{\textbf{Proposition}}
		
		\newtheorem{cor}[lem]{\textbf{Corollary}}

		\newtheorem{theorem}{Theorem}[section]

		\newtheorem{definition}[theorem]{Definition}
		\newtheorem{property}{Property}[section]

\begin{document}
			
\date{ }

\title{Green's function and Large time behavior for the 1-D compressible Euler-Maxwell system}

\author{Boyu Liang$^1$, Mingying Zhong$^{1,2}$\\
\emph
	{\small\it $^1$School of  Mathematics and Information Sciences, Guangxi University, China.} \\
	{\small\it E-mail:\ zgbyliang@163.com}\\
	{\small\it $^2$Center for Applied Mathematical of Guangxi (Guangxi University), Guangxi University, China.} \\
	{\small\it E-mail:\ zhongmingying@gxu.edu.cn}
}

\pagestyle{myheadings}
\markboth{1-D compressible Euler-Maxwell system}%
{B.-Y. Liang,  M.-Y. Zhong}

 \maketitle

\thispagestyle{empty}
			
\begin{abstract}
We study Green's function and the large time behavior of the one-dimensional Euler-Maxwell System with relaxation. Firstly, we construct the Green's function of linearized system and obtain the optimal time decay rates of its solutions. And then, we obtain the global existence and the optimal time decay rates of solutions to the nonlinear system by using Green's function and a suitable energy estimate.

\medskip
 {\bf Key words}. Euler-Maxwell system, Green's function,  spectral analysis, optimal time decay rate.

\medskip
{\bf 2010 Mathematics Subject Classification}. 76X05, 76N10, 35B40
\end{abstract}

			
%
\tableofcontents
			
\section{Introduction}
			
The Euler-Maxwell system is an important model used in plasma physics which describe the dynamics of electrons and ions under the  influence of self-consistent electromagnetic field \cite{R-G,Chen}. When the ions do not move and become a uniform background, the compressible one-fluid Euler-Maxwell system in $\mathbb{R}^3$ reads \cite{Duan,C-J-W,Ue-Ka}:
\be \label{3dEM}
\begin{cases}
	\partial_t n + \nabla \cdot (nu) =0, \\
	\partial_t u + u\cdot \nabla u + \frac{1}{n}\nabla P(n) = -(E+u\times B) -\nu u, \\
	\partial_t E - \nabla \times B = nu, \\
	\partial_t B + \nabla \times E = 0, \\
	\nabla \cdot E = n_\infty-n, \quad \nabla \cdot B = 0,
\end{cases}
\ee
with initial data
\be (n, u, E, B)(0,x)=(n_0, u_0, E_0, B_0)(x), \quad  x \in \mathbb{R}^3,\ee
satisfying
$$
	\nabla \cdot E_0 = n_\infty -n_0,\quad  \nabla \cdot B_0 = 0.
$$
Here, $n=n(t,x) > 0$, $u=(u_1,u_2,u_3)(t,x)$, $E=(E_1,E_2,E_3)(t,x)$, $B=(B_1,B_2,B_3)(t,x)$ over $\{ t>0, x\in \mathbb{R}^3 \}$ are the electron density, electron velocity, electric field and magnetic field, respectively. The positive constants $\nu,~ n_\infty$ denote the velocity relaxation frequency and the equilibrium-charged density of ions. $P$ depending only on $n$ denotes the pressure function with the assumption that $P(n)$ is smooth and $P'(n) > 0$ for $n>0$.	

There has been some mathematical research on the Euler-Maxwell system. For the case of one dimensional, Chen, Jerome and Wang in \cite{C-J-W} proved the global existence of a weak entropy solution by using the Grodunov scheme with the fractional step as well as the compensated compactness method. Later, Jerome \cite{Jerome}  established a local existence theory of smooth solutions for the Cauchy problem in  $\mathbb{R}^3$. As for the global existence, Duan in \cite{Duan} proved a global existence theorem for small initial perturbation, which near a constant equilibrium with zero background magnetic field, in $H^s(\mathbb{R}^3)$ with $s\ge 4$. And he also obtain the optimal time decay rates of the perturbed solution under the regularity assumption $s\ge 13$ by the method of Green's function. In \cite{Ue-Ka}, Ueda and Kawashima obtained a similar result base on the time weighted energy method under the regularity assumption $s\ge 6$. For the case of non-constant background  density, Liu and Zhu in \cite{Liu-Zhu-1} shown that there exist stationary
solutions when the background density is a small perturbation of a positive constant state. Later, Liu, Yin and Zhu in \cite{Liu-Yin-Zhu} studied the asymptotic behaviors between Euler-Maxwell equation and its corresponding Euler-Poisson equation for one dimensional. In the periodic case, i.e. $x\in \mathbb{T}^3$, Peng \cite{Peng} proved the existence of smooth solutions and the solutions converge toward non-constant equilibrium states with zero velocity.

For the two-fluid Euler-Maxwell system, in which the motion of ions is taken into account, Duan, Liu and Zhu in \cite{Duan-Liu-Zhu} established the global existence of solutions and their rates of convergence are obtained for the isentropic case in $\mathbb{R}^3$. Peng in \cite{Peng2} studied the Cauchy and periodic problems for the same model. In \cite{Wang-Feng-Li}, Wang, Feng and Li investigated the nonisentropic case.

Regarding the relaxation limits, Peng, Wang and Gu \cite{Peng-Wang-Gu} proved that the Euler-Maxwell system converges to the drift-diffusion equations locally in time for well-prepared initial data in periodic domain. Subsequently, Hajjej and Peng \cite{H-P} carried out a new asymptotic expansion to refine the results in \cite{Peng-Wang-Gu}. Furthermore, they constructed initial layer corrections to address ill-prepared initial data. Li, Peng and Zhao \cite{Li-Peng-Zhao} utilized the stream function method to derive error estimates between Euler-Maxwell system and the drift-diffusion equations for periodic problem. Recently, Crin-Barat, Peng, Shou and Xu \cite{CB-Peng-Shou-Xu} introduced an effective velocity motivated by Darcy's law, a key step in establishing global-in-time strong convergence in $\mathbb{R}^3$. Additionally, their relaxation procedure employs two initial layer corrections to handle general ill-prepared initial data for strong convergence.
			
In this paper, we consider one-dimensional Euler-Maxwell system, i.e. $x\in\mathbb{R}$. Then, the operator $\nabla$ in \eqref{3dEM} denotes $\nabla=(\partial_x,0,0)^{T}$. Thus, the one-dimensional Euler-Maxwell system can be written as
\be \label{1dEM}
\begin{cases}
	\partial_t n + \partial_x(nu_1) =0, \\
	\partial_t u_1 + u_1\partial_x u_1 + \frac{P'(n)}{n}\partial_x n = -(E_1-u_r\cdot\mathbb{O}_1B_r) -\nu u_1, \\
	\partial_t u_r + u_1\partial_x u_r = -(E_r+u_1\mathbb{O}_1B_r-B_1\mathbb{O}_1u_r) -\nu u_r, \\	
	\partial_t E_r - \mathbb{O}_1\partial_x B_r = nu_r, \\
	\partial_t B_r + \mathbb{O}_1\partial_x E_r = 0, \\
	\partial_t E_1 = nu_1,\quad \partial_x E_1 = n_\infty-n, \\
\partial_t B_1 = 0, \quad   \partial_x B_1=0.
\end{cases}
\ee
Here, $u_r(t,x) = (u_2,u_3)(t,x),~E_r(t,x) = (E_2,E_3)(t,x),~B_r(t,x) = (B_2,B_3)(t,x)$ and $\mathbb{O}_1 =\left( \ba 0 & -1\\ 1 & 0 \ea \right)$. The  last two equations in \eqref{1dEM} show that $B_1$ is a constant. Throughout this paper, we set
\be n_\infty =1, \quad \nu =1, \quad B_1=1. \ee

Notice that the system \eqref{1dEM} admits a constant steady state of the form
\be (n,u,E,B) = (n_\infty,0,0,B_\infty), \ee
where $n_\infty=1 $ and  $B_\infty = (1,0,0)$. Now, we set
$$
	\rho = n - 1, \quad \gamma =P'(1).
$$
Then, the Cauchy problem \eqref{1dEM} can be reformulated as
\be \label{re-em_1}
\begin{aligned}
	\begin{cases}
		\partial_t \rho + \partial_x u_1 = -  \partial_x (\rho u_1), \\
		\partial_t u_1 + \gamma\partial_x \rho +   E_1 +   u_1 =  - u_1 \partial_x u_1 -  \left( \frac{P'(\rho +1)}{\rho+1} - \gamma \right)\partial_x \rho+  u_r\cdot\mathbb{O}_1 B_r, \\
		\partial_t u_r -\mathbb{O}_1u_r  +   u_r +   E_r=  - u_1 \partial_x u_r -  u_1\mathbb{O}_1 B_r,\\
		\partial_t E_r - \mathbb{O}_1\partial_x B_r -   u_r =   \rho u_r, \\
		\partial_t B_r + \mathbb{O}_1\partial_x E_r = 0, \\
		\partial_tE_1-  u_1 = \rho u_1,\quad \partial_x E_1 = -\rho,
	\end{cases}
\end{aligned}
\ee
with initial data
\be \label{re-initial}
	(\rho, u_1, u_r, E_1,E_r, B_r)(0,x) = (\rho_0, u_{10}, u_{r0}, E_{10}, E_{r0}, B_{r0})(x), \quad x \in \mathbb{R},
\ee
satisfying
\be \label{compatable} \partial_x E_{10} = -\rho_0. \ee

\textbf{Notations:} Before stating the main results of this paper, we list some notations. Define the Fourier transform of $f=f(t,x)$ by
$$
\hat{f}(t,k) = \mathcal{F}f(t,k) = \frac{1}{\sqrt{2\pi}}\int_{\mathbb{R}} e^{-ixk}f(t,x) \,dx,
$$
where and throughout this paper we denote $i=\sqrt{-1}$.

We denote by $\| \cdot \|_{L^1}$, $\| \cdot \|_{L^2}$ and $\| \cdot \|_{L^\infty}$ the norms of the function spaces $L^1(\mathbb{R})$, $L^2(\mathbb{R})$ and $L^\infty(\mathbb{R})$ respectively. For any integer $N\ge 1$, the norm of  space $H^N(\mathbb{R})$ is denoted by $\| \cdot \|_{H^N}$ .

\bigskip

In this paper we concern the global existence and time decay rates of solutions to the Cauchy problem \eqref{re-em_1}--\eqref{re-initial}. The main theorem is stated as follows.
\begin{thm} \label{global-tm}
Let $N \ge 2$ and \eqref{compatable} holds. There exists a positive constant $\delta$ such that if
$$
	\|(\rho_0,u_{10},u_{r0},E_{10},E_{r0},B_{r0}) \|_{H^N} \le \delta,
$$
then, the Cauchy problem \eqref{re-em_1}--\eqref{re-initial} has a unique global solution $(\rho,u_1,u_r,E_1,E_r,B_r)$
satisfying
\bma \label{energy-estimate}
	&\quad \| 	(\rho,u_1,u_r,E_1,E_r,B_r) \|_{H^N}^2 + \int_0^t \left( \| (\rho,u_1,u_r)\|_{H^N}^2 + \| (E_1,E_r)\|_{H^{N-1}}^2 + \| \partial_x B_r\|_{H^{N-2}}^2\right) \,d\tau \nnm\\
	&\le C\| (\rho_0,u_{10},u_{r0},E_{10},E_{r0},B_{r0})\|_{H^N}^2, \quad t \ge 0.
\ema
\end{thm}

\begin{thm} \label{Thm-rates}
Let $N \ge 4$ and \eqref{compatable} holds. There exists a positive constant $\delta_0$ such that if
$$
\|(\rho_0,u_{10},u_{r0},E_{10},E_{r0},B_{r0}) \|_{H^N\cap L^1} \le \delta_0,
$$
then, the solution $(\rho,u_1,u_r,E_1,E_r,B_r)$ to the system \eqref{re-em_1}--\eqref{re-initial} satisfies
\bmas
	\| \partial_x^\alpha(\rho,u_1)(t)\|_{L^2} &\le C\delta_0(1+t)^{-\frac54-\frac{\alpha}{2}}, \\
	\| \partial_x^\alpha E_1(t)\|_{L^2} &\le C\delta_0(1+t)^{-\frac54},\\
	\| \partial_x^\alpha (u_r,E_r)(t)\|_{L^2} &\le C\delta_0(1+t)^{-\frac34-\frac{\alpha}{2}}, \\
	\|\partial_x^\alpha B_r(t) \|_{L^2} &\le C\delta_0(1+t)^{-\frac14-\frac{\alpha}{2}},
\emas
for $\alpha=0,1$ and any $t\ge 0$. Moreover, there is a positive constant $d_0$ such that if $\inf\limits_{|k|\le \epsilon}|\hat{B}_{r0}(k)|\ge d_0$, then, the solution $(u_r,E_r,B_r)$ satisfies
\bmas
	C_1 d_0(1+t)^{-\frac{3}{4}-\frac{\alpha}{2}} \le \| \dxa (u_r,E_r)(t)\|_{L^2} &\le C_2\delta_0(1+t)^{-\frac{3}{4}-\frac{\alpha}{2}},\\ 
	C_1 d_0(1+t)^{-\frac{1}{4}-\frac{\alpha}{2}} \le \| \dxa B_r(t)\|_{L^2} &\le C_2\delta_0(1+t)^{-\frac{1}{4}-\frac{\alpha}{2}}, 
\emas
where $C_2\ge C_1>0$ are two constants and $t>0$ is sufficiently large.
\end{thm}

 We now outline the main ideas in the proof of the above theorem. The results in Theorem \ref{Thm-rates} about the time decay rates of the solutions to the Euler-Maxwell system are obtained based on the Green's function of the linearized system. In the one-dimensional case, the linearized Euler-Maxwell system can automatically decouple into a system for $(u_1,E_1)$ and a system for $(u_r,E_r,B_r)$, see \eqref{fEM} and \eqref{eEM}.

And then, by applying Fourier transform and utilizing the right and left eigenvectors of system to construct Green's functions (refer to Lemma \ref{etA}), we obtain the Green's functions $\hat{G}_f$ and $\hat{G}_e$ for the systems \eqref{fEM} and \eqref{eEM} respectively, which have the following pointwise behaviors in Fourier frequency space. For $\hat{G}_f(t,k)$, it holds that for $|k|\le \frac{\sqrt{3\gamma}}{2\gamma}$,
$$
	|\hat{G}_f(t,k)| \le C \left(  \ba 1 & 1  \\  1 & 1 \ea \right)e^{-\frac{t}{2}},
$$
while for $|k|>\frac{\sqrt{3\gamma}}{2\gamma}$,
$$
	|\hat{G}_f(t,k)| \le C \left(  \ba 1 & |k| \\ |k|^{-1} & 1 \ea \right)e^{-\frac{t}{2}}.
$$
For $\hat{G}_e(t,k)$, it holds that for $|k|< \epsilon$,
\bmas
	|\hat{G}_e(t,k)| \le C \left( \ba |k|^2I_1 & |k|^2I_1 & |k|I_1 \\ |k|^2I_1 & |k|^2I_1 & |k|I_1 \\ |k|I_1 & |k|I_1 & I_1 \ea \right)e^{-\frac{k^2t}{2}} + C\left( \ba I_1 & I_1 & |k|I_1 \\ I_1 & I_1 & |k|I_1 \\ |k|I_1 & |k|I_1 & |k|^2I_1 \ea \right)e^{-\frac{t}{8}},
\emas
for $|k|>R$,
\bmas
	|\hat{G}_e(t,k)| \le &\,C \left( \ba I_1 & |k|^{-2}I_1 & |k|^{-1}I_1 \\ |k|^{-2}I_1 & |k|^{-4}I_1 & |k|^{-3}I_1 \\ |k|^{-1}I_1 & |k|^{-3}I_1 & |k|^{-2}I_1 \ea \right)e^{-\frac{t}{2}}\\
&+ C\left( \ba |k|^{-2}I_1 & |k|^{-1}I_1 & |k|^{-1}I_1 \\ |k|^{-1}I_1 & I_1 & I_1 \\ |k|^{-1}I_1 & I_1 & I_1 \ea \right)e^{-\frac{t}{4 k^2}},
\emas
and for $\epsilon\le|k|\le R$,
$$
	|\hat{G}_e(t,k)| \le C \left( \ba I_1 & I_1 & I_1 \\ I_1 & I_1 & I_1 \\ I_1 & I_1 & I_1 \ea \right)e^{-ct}.
$$
Here, $0<\epsilon\ll 1 \ll R < \infty$ are two properly chosen constants, and $I_1= \left( \ba 1 & 1\\ 1 & 1\ea\right)$.

Next, by using Fourier analysis techniques, we can show the time decay rates of the solutions for linearized system, specifically, for any $t\ge0$, if $u_{10}\in H^{\alpha}$ and $E_{10}\in H^{\alpha+1}$, it holds
$$
	\| \partial_x^\alpha (u_1,E_1) \|_{L^2} \le Ce^{-\frac{t}{2}},
$$
and for any $t > 0$ large enough, provided that if $u_{r0}\in H^{2\alpha+2} \cap L^1 $, $E_{r0}\in H^{2\alpha+3} \cap L^1$ and $\inf_{|k|\le \epsilon}|\hat{B}_{r0}(k)| \ge d_0$,  one has
\bmas
	C_1(1+t)^{-\frac34-\frac{\alpha}{2}}\le \| \partial_x^\alpha (u_r,E_r) \|_{L^2} &\le C_2(1+t)^{-\frac34-\frac{\alpha}{2}},\\
	C_1(1+t)^{-\frac14-\frac{\alpha}{2}}\le \| \partial_x^\alpha B_r \|_{L^2} &\le C_2(1+t)^{-\frac14-\frac{\alpha}{2}}.
\emas

Finally,  by applying Duhamel's principle and the virtue of energy estimate \eqref{energy-estimate} in Theorem \ref{global-tm}, one can establish the optimal decay rates on the global solution to the nonlinear Euler-Maxwell system given in Theorem \ref{Thm-rates}.

We should make some remarks on the time decay rates. First, the solutions of $u_r,E_r$ and $B_r$ have optimal decay rates. Not only do they share the same rates as in the linearized case, but their upper and lower bounds also coincide. However, the decay rates of the solutions $\rho,u_1$ and $E_1$ are not optimal. Unlike the exponential decay rates observed in the linearized case, these solutions exhibit polynomial decay rates. This behavior is primarily attributed to the nonlinear effect, specifically caused by the nonlinear term $u_r\cdot\mathbb{O}_1B_r$.

The rest of this paper is organized as follows. In Section 2, we study the spectrum of the linearized Euler-Maxwell system and obtain the optimal time decay rates  of the global solution to the linearized system. In Section 3, we establish the energy inequality and obtain the optimal time decay rates of the global solution to the original nonlinear Euler-Maxwell system.

\setcounter{equation}{0}			
\section{Green's Function}

\subsection{Spectrum Analysis}

In this subsection, we study the spectrum of the linearized system corresponding to \eqref{re-em_1}. Notice that it can be decoupled into two systems, one of which can read as
\be  \label{fEM_origin}
\begin{cases}
	\partial_t\rho + \partial_x u_1 =0\\
	\partial_t u_1 + \gamma\partial_x \rho +   E_1 +   u_1 = 0 ,\\
	\partial_t E_1-  u_1 =0 ,\quad
	\partial_x E_1 = - \rho\\
	(\rho,u_1, E_1)(0,x) = (\rho_0,u_{10}, E_{10}),
\end{cases}
\ee
and the other one takes the form as
\be  \label{eEM}
\begin{cases}
	\partial_t u_r -\mathbb{O}_1 u_r +   u_r +   E_r  = 0 ,\\
	\partial_t E_r - \mathbb{O}_1\partial_x B_r -   u_r =0 ,\\
	\partial_r B_r + \mathbb{O}_1\partial_x E_r =0,\\
	(u_r, E_r, B_r)(0,x) = (u_{r0}, E_{r0}, B_{r0}).
\end{cases}
\ee
Notice that, under the relation $\partial_xE_1 = -\rho$ and \eqref{compatable}, the system \eqref{fEM_origin} can be written as
\be  \label{fEM}
\begin{cases}
	\partial_t u_1 - \gamma\partial_{xx} E_1 +   E_1 +   u_1 = 0 ,\\
	\partial_t E_1-  u_1 =0 ,\\
	(u_1, E_1)(0,x) = (u_{10}, E_{10}).
\end{cases}
\ee
			
After taking the Fourier transform of \eqref{fEM} and \eqref{eEM}, we obtain
\be \label{trans_fEM}
	\partial_t \left(  \ba \hat{u}_1 \\ \hat{E}_1 \ea \right) =
	 -\hat{A}_1\left( \ba \hat{u}_1 \\ \hat{E}_1 \ea \right),
\ee
and
\be \label{trans_eEM}
	\partial_t \left(  \ba \hat{u}_r \\ \hat{E}_r\\ \hat{B}_r \ea \right) =
	-\hat{A}_r \left(  \ba \hat{u}_r \\ \hat{E}_r\\ \hat{B}_r \ea \right),
\ee
where
\be
	\hat{A}_1 = \left( \ba 1  &  1 +\gamma k^2\\ -1  & 0 \ea \right), \quad
	\hat{A}_r =\left( \ba   \mathbb{I}_2 -\mathbb{O}_1 &  \mathbb{I}_2 & 0 \\ - \mathbb{I}_2 & 0 & - ik\mathbb{O}_1 \\ 0 & ik\mathbb{O}_1 & 0 \ea \right),
\ee
with $\mathbb{I}_2 = \left( \ba 1 & 0\\ 0 & 1 \ea \right)$ and $\mathbb{O}_1 =\left( \ba 0 & -1\\ 1 & 0 \ea \right)$. By a direct calculation, we have
\bmas
	\det(\lambda\mathbb{I}_2+\hat{A}_1) &=\lambda^2+ \lambda+ 1+\gamma k^2,\\
	\det(\lambda\mathbb{I}_6+\hat{A}_r) &=(\lambda^3+ \lambda^2+(k^2+ 1)\lambda+  k^2)^2+(\lambda^2+k^2)^2\\
	& =g^{+}(\lambda)\cdot g^{-}(\lambda) ,
\emas
where
$$
 g^{\pm}(\lambda) = \lambda^3+(1 \pm i)\lambda^2+(k^2+ 1)\lambda+(1\pm i)k^2.
$$
			
First, we have the following asymptotic estimates for the eigenvalues of $\hat{A}_1$ corresponding to system \eqref{trans_fEM}.
			
\begin{lem} \label{fEM-eigen-expansion}	
\begin{enumerate} 
	\item[(1)]
		For $|k|\le \frac{\sqrt{3\gamma}}{2\gamma}$, the eigenvalues $\lambda_{\pm}$ of $\hat{A}_1$ satisfy
		\be \label{fem-eigen-low}
			\lambda_{\pm}=-\frac{1}{2} \pm i\frac{\sqrt{3}}{2}\left(1 +\frac{2\gamma}{3}k^2 + O(1)k^4 \right).
		\ee
	\item[(2)]
	For $|k|>\frac{\sqrt{3\gamma}}{2\gamma}$, the eigenvalues $\lambda_{\pm}$ of $\hat{A}_1$ satisfy
		\be \label{fem-eigen-high}
\lambda_{\pm}=-\frac{1}{2} \pm i\sqrt{\gamma}|k|\left(  1 + \frac{3}{8\gamma}k^{-2} + O(1)k^{-4} \right).
		\ee
	\item
		For $k \in \mathbb{R}$, the eigenvalues $\lambda_{\pm}$ of $\hat{A}_1$ satisfy
		\be \label{fem-eigen-mid}
			{\rm Re}\lambda_{\pm} = -\frac12.
		\ee
\end{enumerate}
\end{lem}

\begin{proof}
According to
$$
	\det(\lambda\mathbb{I}_2+\hat{A}_1) =\lambda^2+ \lambda+ 1 +\gamma k^2,
$$
we can solve this equation as
\be  \lambda_{\pm} = \frac{-1 \pm i \sqrt{3 +4\gamma k^2}}{2}. \notag \ee
Then, it is easy to see that
$ {\rm Re}\lambda_{\pm} = -\frac12$ for any $k\in\mathbb{R}$.\\

For $|k|\le \frac{\sqrt{3\gamma}}{2\gamma},$ we take the form
$$
	\lambda_{\pm} = -\frac{1}{2} \pm i\frac{\sqrt{3}}{2}\sqrt{1+\frac{4\gamma}{3}k^2},
$$
and for $|k|>\frac{\sqrt{3\gamma}}{2\gamma}$, one has
$$
	\lambda_{\pm} = -\frac{1}{2} \pm i\sqrt{\gamma}|k|\sqrt{1+\frac{3}{4\gamma k^2}}.
$$
By using Taylor's expansion, we obtain the asymptotic estimate \eqref{fem-eigen-low} and \eqref{fem-eigen-high}.
\end{proof}

Next, we show some properties of the eigenvalues for $\hat{A}_r$ corresponding to system \eqref{trans_eEM}.
\begin{property} \label{property-nonzero}
There are six eigenvalues $\lambda_1, \lambda_2, \lambda_3, \lambda_4, \lambda_5 ,\lambda_6$ for $\hat{A}_r$. Specifically, the eigenvalues $\lambda_1, \lambda_2, \lambda_3$ are the roots of equation
\be
	g^+(\lambda) = \lambda^3 + (1+i)\lambda^2 + (k^2+ 1)\lambda +(1+i)k^2=0,\label{g1}
\ee
and the eigenvalues $\lambda_4, \lambda_5 ,\lambda_6$ are the roots of equation
\be
g^-(\lambda) = \lambda^3 + ( 1-i)\lambda^2 + (k^2+ 1)\lambda +( 1-i)k^2=0.\label{g2}
\ee
Moreover, these eigenvalues have the following properties
\begin{enumerate} 
	\item[(1)]
	$\lambda =0$ is the root of $g^{\pm}(\lambda)=0$ if and only if $k=0$.
	\item $\lambda_1,\lambda_2,\cdots,\lambda_6$ are all complex and ${ \lambda}_{j+3}$ is complex conjugate with $\lambda_j,~j=1,2,3.$ In particular, $\lambda_1=\lambda_4 =0$ for $k=0$.
	\item[(2)]
	$g^{+}(\lambda)=0$ as well as $g^{-}(\lambda)=0$ has no multiple roots.
\end{enumerate}
\end{property}

\begin{proof}
First, we prove the property $(1)$. Substituting $\lambda =0$ into the equation $g^{\pm}(\lambda)=0$, we obtain $(1\pm i)k^2 =0$, which yields $k=0$. On the other hand, when $k=0$, one has $g^{\pm}(\lambda) = \lambda[\lambda^2+(1\pm i)\lambda+1]=0$, which implies that $\lambda=0$ is a root of $g^{\pm}(\lambda)=0$.
	
Now, we want to show the property $(2)$. If $g^+(\lambda)=0$ has a real root $\lambda'$, then  $(\lambda'^2+k^2) = 0$, which leads to a contradiction except for $\lambda'=0$ when $k=0$. And we claim that the roots of $g^+(\lambda)=0$ are not conjugate to each other. Otherwise, we assume that $\lambda_2$ and $\lambda_3$ are conjugate.	For $k\ne 0$, since
\be \lambda_1\lambda_2 + \lambda_1\lambda_3 +\lambda_2\lambda_3 = 2\lambda_1{\rm Re}\lambda_2 + |\lambda_2|^2= k^2 +  1  ,\ee
it yields that $\lambda_1\ne 0$ is real root which leads a contradiction. For $k=0$, we have $\lambda_1=0$, then $\lambda_1+\lambda_2+\lambda_3 = \lambda_2+\lambda_3 = -(1+i)$,  which also leads a contradiction. Thus, the claim  holds.
Since $ g^+(\bar{\lambda})= \overline{g^-(\lambda)}$, we conclude that $\lambda_{j+3}$ must  be conjugate with $\lambda_j$, $j=1,2,3$.

For the third property, if $g^+(\lambda) = 0$ has a multiple root $\lambda_1$, then
\be \label{multi-1} (g^+)'(\lambda_1) = 3\lambda_1^2+2a_2\lambda_1 +a_1 =0, \ee
where $a_2=1+i,a_1=1+k^2$. This together with $g^+(\lambda_1) =\lambda_1^3 +a_2\lambda_1^2+a_1\lambda_1+a_2k^2 = 0$ yields
	\be \label{multi-2} a_2\lambda_1^2+2a_1\lambda_1+3a_2k^2=0. \ee
By \eqref{multi-1} and \eqref{multi-2}, we can solve $\lambda_1$ as
\be \lambda_1 =-\frac{(9k^2-a_1)a_2}{6a_1-2a_2^2}. \ee
Substituting above into \eqref{multi-1} and since $6a_1-2a_2^2 \neq 0$, we obtain
	\be 4a_2^4k^2-a_2^2a_1^2-18a_2^2k^2a_1 + 4a_1^3 + 27a_2^2k^4 =0.\ee
Notice that $a_2^2=2i$ and $a_2^4 = -4$, then for $k\in \mathbb{R}$, it yields
	$$
		0=a_1^3+a_2^4k^2=(k^2+1)^3-4k^2 =k^6 + 3\left(k^2-\frac{1}{6}\right)^2+\frac{33}{36} >0,
	$$
	which leads a contradiction. The argument for $g^{-}(\lambda)=0$ is similar. Therefore, we done with the proof.
\end{proof}

Next, we give the expansions of the eigenvalues of $\hat{A}_r$ by implicit function theorem.
\begin{lem} \label{eEM-eigen-expansion}
\begin{enumerate}
	\item[(1)]
	For $|k| < \epsilon$ sufficiently small, the eigenvalues $\lambda_j$ $(j=1,2,\cdots,6)$ of $\hat{A}_r$ are analytic functions of $k$ and have the following expansions
	\bma \label{eEM-low-expan}
		&\lambda_1 = \bar{\lambda}_4= -k^2 -ik^2 
		+ O(1)k^4,\\
		&\lambda_2 = \bar{\lambda}_5 = -\frac{1+i}{2} + \frac{\sqrt{2i-4}}{2}+\left(\frac{1+i}{2}+\frac{\sqrt{2i-1}}{\sqrt{5}}\right)k^2 +O(1)k^4,\\
		&\lambda_3 = \bar{\lambda}_6 = -\frac{1+i}{2} - \frac{\sqrt{2i-4}}{2}+\left(\frac{1+i}{2}-\frac{\sqrt{2i-1}}{\sqrt{5}}\right)k^2 +O(1)k^4.
	\ema
	\item[(2)]
	For $|k| > R$ sufficiently large, the eigenvalues $\lambda_j$ $(j=1,2,\cdots,6)$  of $\hat{A}_r$  are analytic function of $k$ and have the following expansions
	\bma \label{eEM-high-expan}
		&\lambda_1 = \bar\lambda_4 =-(1+i) +(1+i)k^{-2} + O(1)k^{-3},\\
		&\lambda_2 = \bar\lambda_5 = ik + \frac{i}{2}k^{-1} - \frac{1+i}{2}k^{-2} +O(1)k^{-3},\\
		&\lambda_3 = \bar\lambda_6 =-ik - \frac{i}{2}k^{-1} - \frac{1+i}{2}k^{-2} +O(1)k^{-3}.
	\ema
	\item[(3)]
	For $\epsilon \le |k| \le R$, the eigenvalues $\lambda_{j}$ $(j=1,2,\cdots,6)$ of $\hat{A}_r$  satisfy
	\be {\rm Re}\lambda_j < -c,\ee
where $c>0$ is a constant.
\end{enumerate}
\end{lem}
			
\begin{proof}
Let $r=k^2$, we have
$$
	g^{+}(r,\lambda) = \lambda^3 + (1+i)\lambda^2 + (r+1)\lambda +(1+i)r.
$$
Then the equation
$g^{+}(0,\lambda) = \lambda^3 + (1+i)\lambda^2 + \lambda=0$
has three roots $z_j$ $(j=1,2,3)$ with
$$z_1=0,\quad z_{2}=-\frac{1+i}2+\frac{\sqrt{2i-4}}2,\quad z_3=-\frac{1+i}2-\frac{\sqrt{2i-4}}2.$$
We notice that
\be
g^{+}_{\lambda}(0,0)= 1 \neq 0;  \quad g^{+}_{\lambda}(0,z_2)=(z_2-z_3)z_{2} \neq 0; \quad g^{+}_{\lambda}(0,z_3)= (z_3-z_2)z_{3} \neq 0. \label{lowlam3}
\ee
By the implicit function theorem (\cite{Kaup}, Theorem 8.6), there exist two small constants $\epsilon >0$ and $r_1>0$, for   $r\in D_1:=\{ r\in\mathbb{C}: |r|\le \epsilon\}$, the equation $g^+(r,\lambda)=0$ has a unique solution $\lambda_j(r)$ with $|\lambda_j-z_j|\le r_1$ for each $j=1,2,3.$ And since $g^+(r,\lambda): \mathbb{C} \times \mathbb{C} \to \mathbb{C}$ is analytic, it follows that $\lambda_j(r)$ is an analytic function with respect to $r\in D_1.$ In particular,
\bmas
	&\lambda_1(0) =0, \quad \lambda_1'(0) = -\frac{g^+_r(0,0)}{g^+_{\lambda}(0,0)}  =-(1+i),
\\
	&\lambda_2(0) =z_2, \quad \lambda_2'(0) = -\frac{g^+_r(0,z_2)}{g^+_{\lambda}(0,z_2)}  =\frac{z_3}{(z_2-z_3)z_{2}}=\frac{1+i}{2}+\frac{\sqrt{2i-1}}{\sqrt{5}},\\
	&\lambda_3(0) =z_3, \quad \lambda_3'(0) = -\frac{g^+_r(0,z_3)}{g^+_{\lambda}(0,z_3)}  =\frac{z_2}{(z_3-z_2)z_{3}}=\frac{1+i}{2}-\frac{\sqrt{2i-1}}{\sqrt{5}}.
\emas
Since
$$ \sqrt{2i-4}=(\sqrt5-2)^{\frac12}+i(\sqrt5+2)^{\frac12},$$
it follows that
$$ {\rm Re}\lambda_2(0)=-\frac12(1-(\sqrt5-2)^{\frac12}) <-\frac14. $$
				
Next, we rewrite $g^{+}(\lambda)$ as
$$
	g^{+}(\lambda) = k^3\bigg[ \left(\frac{\lambda}{k}\right)^3 + \frac{1+i}{k}\left(\frac{\lambda}{k}\right)^2+ \left(1+\frac{1}{k^2}\right)\frac{\lambda}{k} +\frac{1+i}{k}\bigg],
$$
Let $s=k^{-1}$ and $\beta=\lambda k^{-1}$, we define an analytic mapping $F: \mathbb{C} \times \mathbb{C} \to \mathbb{C}$ as follows
$$
	F(s,\beta) = \beta^3 + (1+i)s\beta^2 + (1+s^2)\beta +( 1+i)s.
$$
Then the equation
$F(0,\beta) = \beta^3 +\beta =0$
has three roots $\beta=y_j$ with
$$y_1=0,\quad y_2=i,\quad y_3=-i.$$
We have
$$
	F(0,0)=F(0,\pm i)=0,\quad F_{\beta}(0,0)=1, \quad F_{\beta}(0,\pm i)=-2.
$$
Therefore, there exists a large constant $R >0$, for each given $s\in D_2:=\{ s\in\mathbb{C}: |s|\le \frac{1}{R}\}$, the equation $F(s,\beta)=0$ has a unique analytic solution $\beta_j(s)$ with $|\beta_j-y_j|\le r_1$ satisfying $\beta_j(0)=y_j$, for each $j=1,2,3.$  Note that
\bmas
	&\beta'_j(0) =-\frac{1}{F_{\beta}(0,y_j)}F_{s}(0,y_j),\\
	&\beta''_j(0) =-\frac{1}{F_{\beta}(0,y_j)} \left( F_{ss}(0,y_j) + 2F_{\beta s}(0,y_j)\beta'_j(0) + F_{\beta\beta}(0,y_j)\beta'_j(0)^2\right) ,\\
	&\beta'''_j(0) = -\frac{1}{F_{\beta}(0,y_j)}  \big[ F_{sss}(0,y_j) + 3\left(F_{\beta ss}(0,y_j)+F_{\beta\beta s}(0,y_j) \beta'(0)\right)\beta'_j(0) \\
&\qquad\qquad+ 3\left(F_{\beta s}(0,y_j)+F_{\beta\beta}(0,y_j)\beta'_j(0)\right)\beta''_j(0)+F_{\beta\beta\beta}(0,y_j)\beta'_j(0)^3 \big],
\emas
and
\bmas
	&F_{s}=(1+i)\beta ^2+2 s\beta  +  1+i, \quad F_{\beta } =3\beta ^2 +2(1+i)s\beta +1+ s^2,\\
	&F_{\beta s} = 2\beta (1+i)+2 s, \quad F_{ss} = 2 \beta ,\quad F_{\beta \beta } =6\beta +2(1+i)s, \\
	&F_{\beta \beta s} = 2(1+i),\quad F_{\beta ss} =2 , \quad F_{sss} = 0, \quad F_{\beta \beta \beta }=6.
\emas
Thus, we obtain
\bmas
	&\beta _1(0) =0,\quad \beta _1'(0) = -(1+i),\quad \beta _1''(0) =0,\quad \beta _1'''(0) = 6(1+i);\\
	&\beta _2(0) =i,\quad \beta _2'(0) = 0,\quad \beta _2''(0) =i,\quad \beta _2'''(0) = -3(1+i);\\
	&\beta _3(0) =-i,\quad \beta _3'(0) = 0,\quad \beta _3''(0) =-i,\quad \beta _3'''(0) = -3(1+i).
\emas
Therefore,
$$
	\lambda_j(k) =\beta _jk =\beta _j(0)k +\beta _j'(0) +\frac12\beta _j''(0)k^{-1} +\frac16\beta _j'''(0)k^{-2} +O(1)k^{-3}.
$$
				
For $\epsilon \le |k| \le R$, suppose that ${\rm Re}\lambda_j < 0$ is not true, then there exists at least one $|k|\in[\epsilon,R]$ such that $\lambda(k)=id(k)$ with $d(k)\in \mathbb{R}.$ Substituting it into $g^+(\lambda)$ and dividing into real part and imaginary part, we have
$$
\left\{\bal
	d^2-k^2=0,\\
	d^3 + (d^2-k^2)- (k^2+ 1)d=0.
\ea\right.
$$
The first equation leads to $d=\pm k$, which is substituted into the second equation to obtain $k=0$. This contradicts the range of $k$. Hence, ${\rm Re}\lambda_j <-c$ for $c>0$.
\end{proof}

\subsection{Decay estimate of Green's Function }
Before giving the Green's function, we consider an initial value problem of   linear ordinary differential equations with constant coefficients:
\be \label{ode} y'=Ay; \quad y(t)\big|_{t=0}=y(0), \ee
where $y\in \C^n$ is a vector function and $A$ is a $n\times n$ matrix.
			
\begin{definition}
	\begin{enumerate}
		\item[(1)] 	A nonzero vector $x$ is called the right eigenvector of a Matrix $A$ associated with an eigenvalue $\lambda$ of $A$ if $Ax=\lambda x$.
		\item[(2)] 	 A nonzero vector $y$ is called the left eigenvector of a Matrix $A$ associated with an eigenvalue $\lambda$ of $A$ if $\bar{y}^{T}A=\lambda \bar{y}^{T}$, where $\bar{y}$ means conjugate with $y$ and the superscript $``~T~"$ denotes the transposed.
	\end{enumerate}	
\end{definition}
			
The following lemma gives an expression of the solution for \eqref{ode}.
\begin{lem} \label{etA}
	If Matirx $A$ has $n$ distinct linear independent right eigenvectors $v_1,~v_2,\cdots,~v_n$ and $n$ distinct linear independent left eigenvectors $v_1^*,~v_2^*,\cdots,~v_n^*$, their corresponding eigenvalues are $\lambda_1,~\lambda_2,\cdots,~\lambda_n$, respectively. Assume that $(\bar{v}_i^*)^{{T}}v_i \neq 0$, then, the system \eqref{ode} has solution $y(t) = e^{tA}y(0),$ and $e^{tA}$ has the following expression:
\be
	e^{tA} = \sum_{i=1}^n e^{\lambda_it} \frac{v_i \otimes \bar{v}_i^*}{(\bar{v}_i^*)^{{T}}v_i}. \label{ode1}
\ee
\end{lem}
\begin{proof}
By the knowledge of ODE, 
the every solution $y$ of $y'=Ay$ has the form
$$y(t) = \sum_{i=1}^n e^{\lambda_it}a_iv_i,$$
with $a_i$ are constants. For $t=0$, we have $y(0) = \sum_{i=1}^n a_i v_i$. After taking inner product with $v_i^*$, we get
$$
	(\bar{v}_i^*)^{{T}}y(0) = a_i (\bar{v}_i^*)^{{T}}v_i.
$$
In fact, for $i\neq j$,
$$(\bar{v}_j^*)^{{T}}\lambda_iv_i = (\bar{v}_j^*)^{{T}}Av_i = \lambda_j(\bar{v}_j^*)^{{T}}v_i.$$
And then, we have $(\lambda_i-\lambda_j)(\bar{v}_j^*)^{{T}}v_i = 0.$ Since $\lambda_i \neq \lambda_j$, we obtain that $(\bar{v}_j^*)^{{T}}v_i = 0.$
Thus,
$$
	y(t) = \sum_{i=1}^n e^{\lambda_it}\frac{v_i(\bar{v}_i^*)^{{T}}y(0)}{(\bar{v}_i^*)^{{T}}v_i}.
$$
Thus, we obtain \eqref{ode1}.
\end{proof}
		
\begin{lem} \label{fEM-eigenvector}
\begin{enumerate} 
	\item[(1)] The vectors $\psi_+,~\psi_+^*$ are right and left eigenvectors of $\hat{A}_1$ associated with $\lambda_+$, respectively,
	\bma \label{lambda+_eigenvector}
	&\psi_+ = \left(\lambda_+,1\right)^{{T}}, \quad \bar{\psi}_+^* = \left(1,-\lambda_-\right)^{{T}}.
	\ema
	\item[(2)] The vectors $\psi_-,~\psi_-^*$ are right and left eigenvectors of $\hat{A}_1$ associated with $\lambda_-$, respectively,
	\bma \label{lambda-_eigenvector}
	&\psi_- = \left(\lambda_-,1\right)^{{T}}, \quad \bar{\psi}_-^* = \left(1,-\lambda_+\right)^{{T}}.
	\ema
\end{enumerate}
\end{lem}
\begin{proof}
Since $\lambda_+ +\lambda_- =-1 $ and $\lambda_+ \lambda_- = 1+\gamma k^2$, we have
\bmas
	(\lambda_+\mathbb{I}_2 + \hat{A}_1)\psi_+=
	\left( \ba \lambda_+ + 1  &  1+\gamma k^2\\
	-1  & \lambda_+ \ea \right)
	\left( \ba \psi_+^1 \\ \psi_+^2 \ea\right) =
	\left( \ba -\lambda_- & \lambda_+\lambda_-\\
	-1 & \lambda_+ \ea \right)
	\left( \ba \psi_+^1 \\ \psi_+^2 \ea\right) =0,
\emas
and
\bmas
	(\lambda_+\mathbb{I}_2 + \hat{A}_1^{T})\bar{\psi}_+^*=
	\left( \ba \lambda_+ + 1  & -1 \\
	1 +\gamma k^2 & \lambda_+ \ea \right)
	\left( \ba \bar{\psi}_+^{1*} \\ \bar{\psi}_+^{2*} \ea\right) =
	\left( \ba -\lambda_- & -1 \\
	\lambda_+\lambda_- & \lambda_+ \ea \right)
	\left( \ba \bar{\psi}_+^{1*} \\ \bar{\psi}_+^{2*} \ea\right) =0.
\emas
Thus, by taking $\psi_+^2 =1$, we obtain a right eigenvector corresponding to eigenvalue $\lambda_+$
\bmas \psi_+ = \left(\lambda_+,1\right)^{{T}}, \emas
and by taking $\bar{\psi}_+^{1*}=1$, we obtain a left eigenvector
\bmas
	\bar{\psi}_+^* = \left(1,-\lambda_-\right)^{{T}}	.
\emas
Similarly, we can obtain \eqref{lambda-_eigenvector} from $(\lambda_-\mathbb{I}_2 + \hat{A}_1)\psi_-=0$ and $(\lambda_-\mathbb{I}_2 + \hat{A}_1^{T})\bar{\psi}_-^*=0$. Then, we done with this proof.
\end{proof}
			
\begin{lem} \label{eEM-eigenvector}
\begin{enumerate} 
	\item[(1)]
		For $j=1,2,3$, let $\psi_j,~\psi_j^*$ are right and left eigenvectors of $\hat{A}_r$ associated with $\lambda_j$ respectively. Denote $\Lambda_j=\sum\limits_{n=1 \atop n\neq j}^3 \lambda_n$, $e=(1,i)$ and $\bar{e}=(1,-i)$, we have
		\be
			\psi_j = \left( \frac{ 1}{\Lambda_j}e, ~e , ~-\frac{k}{\lambda_j}e \right)^{{T}},  \quad
			\bar{\psi}_j^* = \left( -\frac{ 1}{\Lambda_j}\bar{e} , ~\bar{e} ,  ~\frac{k}{\lambda_j}\bar{e} \right)^{{T}}. \label{lambda-j-eigenvector}
		\ee
	\item[(2)]
		For $l=4,5,6,$ let $\psi_l,~\psi_l^*$ are right and left eigenvectors of $\hat{A}_r$ associated with $\lambda_l$ respectively. Denote $\Omega_l=\sum\limits_{m=4 \atop m\neq l}^6 \lambda_m$, $e=(1,i)$ and $\bar{e}=(1,-i)$, we have
		\be
			\psi_l = \left( \frac{ 1}{\Omega_l}\bar{e}, ~\bar{e} , ~\frac{k}{\lambda_l}\bar{e} \right)^{{T}}, \quad
			\bar{\psi}_l^* = \left( -\frac{ 1}{\Omega_l}e, ~e , ~-\frac{k}{\lambda_l}e \right)^{{T}}. \label{lambda-l-eigenvector}
		\ee
\end{enumerate}
\end{lem}

\begin{proof}
First, we compute the right eigenvector of $\hat{A}_r$ associate with $\lambda_j,j=1,2,3$. Denote that
$\psi_j=(a_j,b_j,c_j)^T$, where $a_j,b_j,c_j \in \mathbb{C}^2,$ and we have
$$
	(\lambda_j\mathbb{I}_6 + \hat{A}_r)\psi_j =
	\begin{pmatrix}
		(\lambda_j+ 1)\mathbb{I}_2-\mathbb{O}_1  & \mathbb{I}_2 & 0  \\
		-\mathbb{I}_2 & \lambda_j\mathbb{I}_2 & -ik\mathbb{O}_1 \\
		0 & ik\mathbb{O}_1 & \lambda_j\mathbb{I}_2	\end{pmatrix}
	\begin{pmatrix}
		a_j \\ b_j \\ c_j
	\end{pmatrix} = \left( \ba 0 \\ 0 \\ 0 \ea \right).
$$
From the second and third rows, we obtian
\bma
&c_j = -\frac{ik}{\lambda_j}\mathbb{O}_1b_j,\\
&a_j=\lambda_jb_j-ik\mathbb{O}_1c_j =\left(\lambda_j+\frac{k^2}{\lambda_j}\right)b_j. \label{a_j}
\ema
For $j=1$, we calculate that
\bmas
	\lambda_1+\frac{k^2}{\lambda_1} &= \frac{(\lambda_1^2+ k^2)(\lambda_2+\lambda_3)  }{\lambda_1(\lambda_2+\lambda_3)} =-\frac{(\lambda_1^2+ k^2) (\lambda_1+1+i)}{\lambda_1(\lambda_2+\lambda_3)}\\
	&=-\frac{\lambda_1^3 +(1+i)\lambda_1^2 +k^2\lambda_1 +(1+i)k^2}{\lambda_1(\lambda_2+\lambda_3)} = \frac{1}{\lambda_2+\lambda_3}.
\emas
where we had used
$	\lambda_1+\lambda_2 + \lambda_3 = -( 1+i) $ and 
$g^+(\lambda_1)=\lambda_1^3 + (1+i)\lambda_1^2 + (k^2+ 1)\lambda_1 +(1+i)k^2=0.$

Similarly, we  have
$$
	\lambda_2+\frac{k^2}{\lambda_2} = \frac{1}{\lambda_1+\lambda_3}, \quad \lambda_3+\frac{k^2}{\lambda_3} = \frac{1}{\lambda_1+\lambda_2}.
$$
Then, \eqref{a_j} becomes
$$
	a_j = \frac{1}{\Lambda_j}b_j,
$$
where $\Lambda_j = \sum\limits_{n=1 \atop n \neq j}^3\lambda_n$ for $j=1,2,3$.
Next, from the first row, we have
$$
	0= (\lambda_j+1)a_j-\mathbb{O}_1a_j+b_j=(\lambda_j+1)\left(\lambda_j+\frac{k^2}{\lambda_j}\right)b_j - \left(\lambda_j+\frac{k^2}{\lambda_j}\right)\mathbb{O}_1b_j + b_j.
$$
Multiplying it by $\Lambda_j$, it follows
$$
	 \mathbb{O}_1b_j=(\lambda_j+1+\Lambda_j)b_j= -ib_j.
$$
Therefore, we take
$$
	b_j=(1,i) =:e,
$$
and obtain the following right eigenvector of $\hat{A}_r$ associate with $\lambda_j$:
$$
	\psi_j = \left(  \frac{1}{\Lambda_j}e, ~e,  ~-\frac{k}{\lambda_j}e \right)^{T}.
$$
			
Denote the left eigenvector by $\psi_j^*=(a_j^{*},b_j^{*},c_j^{*})^T$. We have
\bmas
&(\lambda_j\mathbb{I}_6 + \hat{A}_r^{T})\bar{\psi}_j^* =
\begin{pmatrix}
	(\lambda_j+ 1)\mathbb{I}_2+\mathbb{O}_1  & -\mathbb{I}_2 & 0  \\
	\mathbb{I}_2 & \lambda_j\mathbb{I}_2 & -ik\mathbb{O}_1 \\
	0 & ik\mathbb{O}_1 & \lambda_j\mathbb{I}_2	
\end{pmatrix}
\begin{pmatrix}
	\bar{a}_j^* \\ \bar{b}_j^* \\ \bar{c}_j^*
\end{pmatrix}=0.
\emas
Then, after taking $\bar{b}_j^{*} =(1,-i) = \bar{e}$, the left eigenvector of $\hat{A}_r$ associate with $\lambda_j$ has the following form
$$
	\bar{\psi}_j^* = \left(  -\frac{1}{\Lambda_j}\bar{e}, ~\bar{e},  ~\frac{k}{\lambda_j}\bar{e} \right)^{T}.
$$
By the same argument, we can obtain \eqref{lambda-l-eigenvector} corresponding to $\lambda_l$ $ (l=4,5,6)$. The proof is completed.
\end{proof}
			
Thus, by the lemma \ref{etA}, we obtain the following Green's functions of \eqref{trans_fEM} and \eqref{trans_eEM}.
\begin{prop} \label{f-Green}
Let $\hat{U}_{1} = (\hat{u}_{1},\hat{E}_{1})^{T}$ be a solution of the system \eqref{trans_fEM} for all $t\ge 0,~k \in \mathbb{R}$. Then,
\be \label{solution-U1}
	\left(
	\begin{array}{c}
	\hat{u}_{1}(t,k)\\
	\hat{E}_{1}(t,k)
	\end{array}
	\right)
	=
	\hat{G}_f
	\left(
	\begin{array}{c}
	\hat{u}_{10}\\
	\hat{E}_{10}
	\end{array}
	\right).
\ee
Here $\hat{G}_f = \hat{G}_f(t,k)$ is given by
\be \label{Gf}
	\setlength{\arraycolsep}{10pt}
	\hat{G}_f =
	\left(
	\begin{array}{ccc}
	\frac{\lambda_+e^{\lambda_+t}-\lambda_-e^{\lambda_-t}}{\lambda_+-\lambda_-} & -(1+\gamma k^2) \frac{e^{\lambda_+t}-e^{\lambda_-t}}{\lambda_+-\lambda_-} \\
	  \frac{e^{\lambda_+t}-e^{\lambda_-t}}{\lambda_+-\lambda_-} & \frac{\lambda_+e^{\lambda_-t}-\lambda_-e^{\lambda_+t}}{\lambda_+-\lambda_-}
	\end{array}
	\right).
\ee
\end{prop}

\begin{proof}
According to Lemma \ref{etA} and Lemma \ref{fEM-eigenvector}, we have
\bmas
	\hat{G}_f &= \frac{\psi_+\otimes\bar{\psi}_+^*}{(\bar{\psi}_+^*)^{T}\psi_+}e^{\lambda_+t} + \frac{\psi_-\otimes\bar{\psi}_-^*}{(\bar{\psi}_-^*)^{T}\psi_-}e^{\lambda_-t}\\
	&=\frac{1}{\lambda_+-\lambda_-}\left( \ba \lambda_+ & -(1 +\gamma k^2)\\ 1 & -\lambda_-  \ea \right)e^{\lambda_+t} - \frac{1}{\lambda_+-\lambda_-}\left( \ba \lambda_- & -(1 +\gamma k^2)\\ 1 & -\lambda_+  \ea \right)e^{\lambda_-t},
\emas
which proves \eqref{Gf}.
\end{proof}
			
\begin{prop} \label{e-Green}
Let $\hat{U}_{r} = (\hat{u}_{r},\hat{E}_{r},\hat{B}_r)^{T}$ be a solution of the system \eqref{trans_eEM} for all $t\ge 0,~k \in \mathbb{R}$. Then,
\be
	\left(
	\begin{array}{c}
	\hat{u}_{r}(t,k)\\
	\hat{E}_{r}(t,k)\\
	\hat{B}_r(t,k)
	\end{array}
	\right)
	=
	\hat{G}_e
	\left(
	\begin{array}{c}
	\hat{u}_{r0}\\
	\hat{E}_{r0}\\
	\hat{B}_{r0}
	\end{array}
	\right).
\ee
Here $\hat{G}_e = \sum\limits_{j=1}^3P_je^{\lambda_j t} + \sum\limits_{l=4}^6P_le^{\lambda_l t}$ with
\bma
	&P_j= -\frac{\lambda_j\Lambda_j}{2\prod\limits_{k=1 \atop k\neq j}^3(\lambda_j-\lambda_k)}
	\left(  \ba
	-\frac{1}{\Lambda_j^2}\left( \mathbb{I}_2+i\mathbb{O}_1 \right) & \frac{1}{\Lambda_j}\left( \mathbb{I}_2+i\mathbb{O}_1 \right) & \frac{k}{\lambda_j\Lambda_j}\left( \mathbb{I}_2+i\mathbb{O}_1 \right)\\
	-\frac{1}{\Lambda_j}\left( \mathbb{I}_2+i\mathbb{O}_1 \right) & \mathbb{I}_2+i\mathbb{O}_1 & \frac{k}{\lambda_j}\left( \mathbb{I}_2+i\mathbb{O}_1 \right) \\
	\frac{k}{\lambda_j\Lambda_j}\left( \mathbb{I}_2+i\mathbb{O}_1 \right) &  -\frac{k}{\lambda_j}\left( \mathbb{I}_2+i\mathbb{O}_1 \right) &  -\frac{k^2}{\lambda_j^2}\left( \mathbb{I}_2+i\mathbb{O}_1 \right)
    \ea \right), \label{Pj} \\
	&P_l= -\frac{\lambda_l\Omega_l}{2\prod\limits_{k=4 \atop k\neq l}^6(\lambda_l-\lambda_k)}
	\left(  \ba
	-\frac{1}{\Omega_l^2}\left( \mathbb{I}_2-i\mathbb{O}_1 \right) & \frac{1}{\Omega_l}\left( \mathbb{I}_2-i\mathbb{O}_1 \right) & -\frac{k}{\lambda_l\Omega_l}\left( \mathbb{I}_2-i\mathbb{O}_1 \right)\\
	-\frac{1}{\Omega_l}\left( \mathbb{I}_2-i\mathbb{O}_1 \right) & \mathbb{I}_2-i\mathbb{O}_1 & -\frac{k}{\lambda_l}\left( \mathbb{I}_2-i\mathbb{O}_1 \right) \\
	-\frac{k}{\lambda_l\Omega_l}\left( \mathbb{I}_2-i\mathbb{O}_1 \right) &  \frac{k}{\lambda_l}\left( \mathbb{I}_2-i\mathbb{O}_1 \right) &  -\frac{k^2}{\lambda_l^2}\left( \mathbb{I}_2-i\mathbb{O}_1 \right)
	\ea \right), \label{Pl}
\ema
where $\Lambda_j=\sum\limits_{n=1 \atop n\neq j}^3 \lambda_n$, $\Omega_l=\sum\limits_{m=4 \atop m\neq l}^6 \lambda_m$,  $\mathbb{I}_2=\left( \ba 1 & 0\\ 0 & 1 \ea \right) $ and $\mathbb{O}_1 = \left( \ba 0 & -1 \\ 1 & 0 \ea \right).$
\end{prop}

\begin{proof}
According to Lemma \ref{etA} and Lemma \ref{eEM-eigenvector}, we have
$$
	\hat{G}_e =\sum_{j=1}^{3}\frac{\psi_j\otimes\bar{\psi}_j^*}{(\bar{\psi}_j^*)^{T}\psi_j}e^{\lambda_jt} + \sum_{l=4}^{6}\frac{\psi_l\otimes\bar{\psi}_l^*}{(\bar{\psi}_l^*)^{T}\psi_l}e^{\lambda_lt}:=\sum\limits_{j=1}^3P_je^{\lambda_j t} + \sum\limits_{l=4}^6P_le^{\lambda_l t}.
$$
For $j=1$, by using the relation $\frac{1}{\lambda_2+\lambda_3} = \lambda_1 +\frac{k^2}{\lambda_1},$ we have
\bmas
(\bar{\psi}_1^{*})^{T}\psi_1
&=-\frac{2}{\lambda_2+\lambda_3}\left[\frac{1}{\lambda_2+\lambda_3}-(\lambda_2+\lambda_3)+\frac{(\lambda_2+\lambda_3)k^2}{\lambda_1^2}\right]\\
&=-\frac{2}{\lambda_2+\lambda_3}\left[\lambda_1-(\lambda_2+\lambda_3)+\frac{\lambda_1\lambda_2\lambda_3}{\lambda_1^2}\right] =-\frac{2(\lambda_1-\lambda_2)(\lambda_1-\lambda_3)}{\lambda_1(\lambda_2+\lambda_3)}.
\emas
Similarly for $j=2,3$ and $l=4,5,6$,  we have
$$
(\bar{\psi}_j^*)^{T}\psi_j=-\frac{2\prod\limits_{n=1 \atop n \neq j}^3(\lambda_j-\lambda_n)}{\lambda_j\Lambda_j}, \quad (\bar{\psi}_l^*)^{T}\psi_l=-\frac{2\prod\limits_{m=4 \atop m\neq l}^6(\lambda_l-\lambda_m)}{\lambda_l\Omega_l}.
$$
Therefore, \eqref{Pj} and \eqref{Pl} can be obtained from above. And then, we complete this proof.
\end{proof}
			
Next, we consider the pointwise behaviors of $\hat{G}_f$ and $\hat{G}_e$.
			
\begin{prop} \label{upper-fEM}
The elements of $\hat{G}_f(t,k)$ have the following upper bounds for $t\ge 0$. For $|k| \le \frac{\sqrt{3\gamma}}{2\gamma}$, it holds
\be
	|\hat{G}_f(t,k)| \le C \left(  \ba 1 & 1  \\  1 & 1 \ea \right)e^{-\frac{t}{2}},
\ee
and for $|k| > \frac{\sqrt{3\gamma}}{2\gamma}$,  it holds
\be
	|\hat{G}_f(t,k)| \le C \left(  \ba 1 & |k| \\ |k|^{-1} & 1 \ea \right)e^{-\frac{t}{2}}.
\ee
Here, $|\hat{G}_f(t,k)|$ denotes that each element of the matrix $\hat{G}_f(t,k)$ takes the absolute value.
\end{prop}

\begin{proof}
 We denote $\eta = \frac{\sqrt{3\gamma}}{2\gamma}. $ Then from Lemma \ref{fEM-eigen-expansion}, we have
$$
	|\lambda_{\pm}|= \left| -\frac{1}{2} \pm \frac{i}{2}\sqrt{3+4\gamma k^2}\right|
	=\left\{\bal
	O(1), & |k| \le \eta,\\
	O(1)|k|, & |k| > \eta,
	\ea\right.
$$
and
$$
	|\lambda_{+}-\lambda_-| = \sqrt{3+4\gamma k^2}
	=\left\{\bal
	 O(1), & |k| \le \eta,\\
	 O(1)|k|, & |k| >\eta.
	\ea\right.
$$
Therefore, by Proposition \ref{f-Green}, we have the upper bounds of $\hat{G}_f$ as follows:
\bmas
	&|\hat{G}_f^{11}|,|\hat{G}_f^{22}| \le \left( \frac{|\lambda_+|}{|\lambda_+-\lambda_-|}+\frac{|\lambda_-|}{|\lambda_+-\lambda_-|}\right)e^{-\frac{t}{2}} \le Ce^{-\frac{t}{2}}, \quad k \in \mathbb{R}, \\
	&|\hat{G}_f^{12}| \le \frac{2\left|1+\gamma k^2\right|}{|\lambda_+-\lambda_-|}e^{-\frac{t}{2}} \le \begin{cases}
		Ce^{-\frac{t}{2}}, & |k| \le \eta,\\
		C|k|e^{-\frac{t}{2}}, & |k| > \eta,
	\end{cases}\\
	&|\hat{G}_f^{21}|\le \frac{2}{|\lambda_+-\lambda_-|}e^{-\frac{t}{2}} \le
	\begin{cases}
		Ce^{-\frac{t}{2}},\quad & |k| \le \eta, \\
		C|k|^{-1}e^{-\frac{t}{2}}, & |k| > \eta.
	\end{cases}
\emas
Then, we complete the proof.
\end{proof}

\begin{prop} \label{upper_eEM}
The elements of $\hat{G}_e(t,k)$ have the following upper bounds for $t\ge 0$. For $|k|< \epsilon$,
\bmas
	|\hat{G}_e(t,k)| \le &\,C \left( \ba |k|^2I_1 & |k|^2I_1 & |k|I_1 \\ |k|^2I_1 & |k|^2I_1 & |k|I_1 \\ |k|I_1 & |k|I_1 & I_1 \ea \right)e^{-\frac{k^2t}{2}}\\
	&+ C\left( \ba I_1 & I_1 & |k|I_1 \\ I_1 & I_1 & |k|I_1 \\ |k|I_1 & |k|I_1 & |k|^2I_1 \ea \right)e^{-\frac{t}{8}}.
\emas
For $|k|>R$,
\bmas
	|\hat{G}_e(t,k)| \le &\,C \left( \ba I_1 & |k|^{-2}I_1 & |k|^{-1}I_1 \\ |k|^{-2}I_1 & |k|^{-4}I_1 & |k|^{-3}I_1 \\ |k|^{-1}I_1 & |k|^{-3}I_1 & |k|^{-2}I_1 \ea \right)e^{-\frac{t}{2}}\\
	&+ C\left( \ba |k|^{-2}I_1 & |k|^{-1}I_1 & |k|^{-1}I_1 \\ |k|^{-1}I_1 & I_1 & I_1 \\ |k|^{-1}I_1 & I_1 & I_1 \ea \right)e^{-\frac{t}{4 k^2}}.
\emas
For $\epsilon\le|k|\le R$,
$$
	|\hat{G}_e(t,k)| \le C \left( \ba I_1 & I_1 & I_1 \\ I_1 & I_1 & I_1 \\ I_1 & I_1 & I_1 \ea \right)e^{-ct}.
$$
Here, $|\hat{G}_e(t,k)|$ denotes  the absolute value of each element of the matrix $\hat{G}_e(t,k)$ and  $I_1 = \left( \ba 1 & 1 \\ 1 & 1 \ea \right).$
\end{prop}

\begin{proof}
From Lemma \ref{eEM-eigen-expansion}, for $|k| < \epsilon$ sufficiently small, $\lambda_j,~j=1,2,3$ have the following expansions:
\bmas
	\lambda_1&= -k^2+O(k^4)-i\left[k^2+O(k^4)\right],\\
	\lambda_2 &= -\frac{1-a_1}{2} + \frac{\sqrt{5}+2a_2}{2\sqrt{5}}k^2 +O(k^4) + i\left[ \frac{b_1-1}{2}+\frac{1+2b_2}{2\sqrt{5}}k^2+O(k^4)\right],\\
	\lambda_3 &= -\frac{1+a_1}{2} + \frac{\sqrt{5}-2a_2}{2\sqrt{5}}k^2 +O(k^4) - i\left[ \frac{b_1+1}{2}-\frac{1-2b_2}{2\sqrt{5}}k^2+O(k^4)\right],
\emas
where
\bmas
&a_1={\rm Re}\sqrt{2i-4}=(\sqrt{5}-2)^{\frac12}, \quad b_1={\rm Im}\sqrt{2i-4}=(\sqrt{5}+2)^{\frac12},\\
&a_2={\rm Re}\sqrt{2i-1}=\Big(\frac{\sqrt{5}-1}{2}\Big)^{\frac12}, \quad b_2={\rm Im}\sqrt{2i-1}= \Big(\frac{\sqrt{5}+1}{2}\Big)^{\frac12}.
\emas
Thus,
\bmas
	\lambda_1 \pm \lambda_2 &= \mp \frac{1-a_1}{2} +O(k^2) \pm i\left[ \frac{b_1-1}{2} + O(k^2)\right],\\
	\lambda_1 \pm \lambda_3 &= \mp \frac{1+a_1}{2} +O(k^2) \mp i\left[ \frac{b_1+1}{2} + O(k^2)\right],\\
	\lambda_2+\lambda_3 &= -1+O(k^2) +i[-1+O(k^2)], \quad
	\lambda_2-\lambda_3 = a_1 +O(k^2) +i[b_1+O(k^2)].
\emas
And for $|k| > R $ is sufficiently large, one has
\bmas
	\lambda_1 &=-1 +k^{-2} + O(k^{-3}) + i\left[ -1 +k^{-2} +O(k^{-3})\right],\\
	\lambda_2 &= -\frac12k^{-2} +O(k^{-3}) + ik\left[ 1 +\frac12k^{-2} +O(k^{-3}) \right],\\
	\lambda_3 &= -\frac12k^{-2} +O(k^{-3}) - ik\left[ 1 +\frac12k^{-2} +O(k^{-3})\right].
\emas
Thus,
\bmas
	\lambda_1 \pm \lambda_2 &= -1 +O(k^{-2}) + i\left[ -1\pm k +O(k^{-1})\right],\\
	\lambda_1 \pm \lambda_3 &= -1 +O(k^{-2}) + i\left[ -1 \mp k +O(k^{-1})\right],\\
	\lambda_2+\lambda_3 &= -k^{-2} +O(k^{-3}) + iO(k^{-2}), \quad
	\lambda_2-\lambda_3 = O(k^{-3}) + ik\left[2 + O(k^{-2})\right].
\emas
From above  and Proposition \ref{e-Green},  we have
\bmas
	\bigg| \sum_{j=1}^3P_j^{11}e^{\lambda_jt}\bigg| &\le \sum_{j=1}^3\frac{|\lambda_j||e^{\lambda_jt}|I_1}{2|\Lambda_j|\prod\limits_{k=1 \atop k\neq j}^3|\lambda_j-\lambda_k|}
	\le C \begin{cases}
		|k|^2I_1e^{-\frac{k^2t}{2}} + I_1e^{-\frac{t}{8}}, &  |k|< \epsilon,\\
		I_1e^{-\frac{t}{2}}+|k|^{-2}I_1e^{-\frac{t}{4k^2}}, &  |k|>R.
	\end{cases}\\
	\bigg| \sum_{j=1}^3P_j^{22}e^{\lambda_jt}\bigg| &\le \sum_{j=1}^3\frac{|\lambda_j||\Lambda_j||e^{\lambda_jt}|I_1}{2\prod\limits_{k=1 \atop k\neq j}^3|\lambda_j-\lambda_k|}
	\le C \begin{cases}
		|k|^2I_1e^{-\frac{k^2t}{2}} + I_1e^{-\frac{t}{8}}, &  |k|< \epsilon,\\
		|k|^{-4}I_1e^{-\frac{t}{2}}+I_1e^{-\frac{t}{4k^2}}, &  |k|>R.
	\end{cases}\\
	\bigg| \sum_{j=1}^3P_j^{33}e^{\lambda_jt}\bigg| &\le \sum_{j=1}^3\frac{k^2}{2|\lambda_j|}\frac{|\Lambda_j||e^{\lambda_jt}|I_1}{\prod\limits_{k=1 \atop k\neq j}^3|\lambda_j-\lambda_k|}
	\le C \begin{cases}
		I_1e^{-\frac{k^2t}{2}} + |k|^2I_1e^{-\frac{t}{8}}, &  |k|< \epsilon,\\
		|k|^{-2}I_1e^{-\frac{t}{2}}+I_1e^{-\frac{t}{4k^2}}, &  |k|>R,
	\end{cases}
\emas
where we had used
\bmas
&{\rm Re}\lambda_1\le -\frac{k^2}2, \quad {\rm Re}\lambda_j\le -\frac{1}8,\quad j=2,3,~~ |k| \le \epsilon,\\
 &{\rm Re}\lambda_1 \le -\frac12, \quad {\rm Re}\lambda_j \le -\frac{1}{4k^2},\quad j=2,3,~~|k| >R.
 \emas
And
\bmas
	\bigg| \sum_{j=1}^3P_j^{12}e^{\lambda_jt}\bigg|,\bigg| \sum_{j=1}^3P_j^{21}e^{\lambda_jt}\bigg| &\le \sum_{j=1}^3\frac{|\lambda_j||e^{\lambda_jt}|I_1}{2\prod\limits_{k=1 \atop k\neq j}^3|\lambda_j-\lambda_k|}
	\le C \begin{cases}
		|k|^2I_1e^{-\frac{k^2t}{2}} + I_1e^{-\frac{t}{8}}, &  |k|< \epsilon,\\
		|k|^{-2}I_1e^{-\frac{t}{2}}+|k|^{-1}I_1e^{-\frac{t}{4k^2}}, &  |k|>R.
	\end{cases}\\
	\bigg| \sum_{j=1}^3P_j^{13}e^{\lambda_jt}\bigg|, \bigg|\sum_{j=1}^3P_j^{31}e^{\lambda_jt}\bigg| &\le \sum_{j=1}^3\frac{|k||e^{\lambda_jt}|I_1}{2\prod\limits_{k=1 \atop k\neq j}^3|\lambda_j-\lambda_k|}
	\le C \begin{cases}
		|k|I_1e^{-\frac{k^2t}{2}} + |k|I_1e^{-\frac{t}{8}}, &  |k|< \epsilon,\\
		|k|^{-1}I_1e^{-\frac{t}{2}}+|k|^{-1}I_1e^{-\frac{t}{4k^2}}, &  |k|>R.
	\end{cases}\\
	\bigg| \sum_{j=1}^3P_j^{23}e^{\lambda_jt}\bigg|, \bigg| \sum_{j=1}^3P_j^{32}e^{\lambda_jt}\bigg| &\le \sum_{j=1}^3\frac{|k||\Lambda_j||e^{\lambda_jt}|I_1}{2\prod\limits_{k=1 \atop k\neq j}^3|\lambda_j-\lambda_k|}
	\le C \begin{cases}
		|k|I_1e^{-\frac{k^2t}{2}} + |k|I_1e^{-\frac{t}{8}}, &\quad |k|< \epsilon,\\
		|k|^{-3}I_1e^{-\frac{t}{2}}+I_1e^{-\frac{t}{4k^2}}, &\quad |k|>R.
	\end{cases}
\emas
 By the third term of Property \ref{property-nonzero}, one can see that $|P_j|$ $(j=1,2,3)$ are bounded for $\epsilon \le |k| \le R$. And using Lemma \ref{eEM-eigen-expansion}, for $\epsilon \le |k| \le R$, we obtain
$$
	\bigg| \sum_{j=1}^3P_je^{\lambda_jt}\bigg| \le C\left( \ba I_1 & I_1 & I_1 \\ I_1 & I_1 & I_1 \\ I_1 & I_1 & I_1 \ea\right)e^{-ct}.
$$
Since $\lambda_j$ $(j=4,5,6)$ are conjugate with $\lambda_{j-3}$, $\sum_{l=4}^6P_le^{\lambda_lt}$ have the same estimates as $\sum_{j=1}^3P_je^{\lambda_jt}$. Thus, we complete this proof.
\end{proof}			
	
\begin{lem} \label{Gf-convolution}
Let the function $f=f(x)$ and its derivatives belong to $L^2(\mathbb{R})$ . Then, the following estimates hold 
\bma
	&\| \partial_x^\alpha(G_f^{ii}\ast f)\|_{L^2} \le Ce^{-\frac{t}{2}}\| \partial_x^\alpha f\|_{L^2}, \quad i=1,2,\\
	&\| \partial_x^\alpha(G_f^{12}\ast f)\|_{L^2} \le Ce^{-\frac{t}{2}}( \| \partial_x^\alpha f\|_{L^2} + \| \partial_x^{\alpha+1} f\|_{L^2}),\\
	&\| \partial_x^\alpha(G_f^{21}\ast f)\|_{L^2} \le Ce^{-\frac{t}{2}} \| \partial_x^{(\alpha-1)_+} f\|_{L^2},
\ema
where $(\alpha-1)_+ = \max{\{0,\alpha-1\}}$.
\end{lem}
\begin{proof}
By Plancherel's theorem and from Proposition \ref{upper-fEM}, for $i=1,2$, we have
\bmas
	\| \partial_x^\alpha (G_f^{ii}\ast f)\|_{L^2}^2 &= \int_{|k|\le \eta}  |(ik)^\alpha \hat{G}_f^{ii}\hat{f} |^2\,dk + \int_{|k|> \eta} |(ik)^\alpha \hat{G}_f^{ii}\hat{f} |^2\,dk\\
	&\le Ce^{-t}\int_{\R}|k|^\alpha| \hat{f}|^2\,dk \le Ce^{-t} \| \partial^\alpha_x f\|_{L^2}^2,
\emas
where $\eta = \frac{\sqrt{3\gamma}}{2\gamma}.$ And, we have
\bmas
	\| \partial_x^\alpha (G_f^{12}\ast f)\|_{L^2}^2 &= \int_{|k|\le \eta}| (ik)^\alpha \hat{G}_f^{12}\hat{f}|^2\,dk + \int_{|k|> \eta}| (ik)^\alpha \hat{G}_f^{12}\hat{f}|^2\,dk\\
	&\le Ce^{-t}\int_{|k|\le \eta}|(ik)^\alpha \hat{f} |^2\,dk+ Ce^{-t}\int_{|k|> \eta}|(ik)^\alpha|k|^{2} \hat{f} |^2\,dk\\
	&\le Ce^{-t}( \| \partial_x^\alpha f\|_{L^2}^2 + \| \partial^{\alpha+1}_x f\|_{L^2}^2),\\
	\| \partial_x^\alpha (G_f^{21}\ast f)\|_{L^2}^2 &= \int_{|k|\le \eta}| (ik)^\alpha \hat{G}_f^{21}\hat{f}|^2\,dk + \int_{|k|> \eta}| (ik)^\alpha \hat{G}_f^{21}\hat{f}|^2\,dk\\
	&\le Ce^{-t}\int_{|k|\le \eta}|(ik)^\alpha \hat{f} |^2\,dk + Ce^{-t}\int_{|k|> \eta}|(ik)^\alpha|k|^{-1} \hat{f} |^2\,dk\\
	&\le Ce^{-t} \| \partial^{(\alpha-1)_+}_x f\|_{L^2}^2,
\emas
where $(\alpha-1)_+ = \max{\{ 0,\alpha-1\}}$. Thus, we done with this proof.
\end{proof}

\begin{cor} \label{U1_linear_time}
Suppose $U_1 =(u_1,E_1)^T$ is the solution to the Cauchy problem \eqref{fEM}. Then, for any integer $\alpha \ge 0$, 
we have the following time decay rates:
\bma
	&\| \partial_x^\alpha u_1 \|_{L^2} \le Ce^{-\frac{t}{2}} \| (\partial^\alpha_x u_{10}, \partial_x^\alpha E_{10}, \partial^{\alpha+1}_xE_{10})\|_{L^2},\\
	&\| \partial_x^\alpha E_1 \|_{L^2} \le Ce^{-\frac{t}{2}} \| (\partial^{(\alpha-1)_+}_x u_{10}, \partial^{\alpha}_xE_{10})\|_{L^2},
\ema
where $C$ is a positive constant and $(\alpha-1)_+ = \max{\{0,\alpha-1\}}$.
\end{cor}

\begin{lem} \label{Ge-convolution}
Assume that vector function $g=(g_1(x),g_2(x))$, where $g_1(x),g_2(x)$ belong to $L^1(\mathbb{R})\cap L^2(\mathbb{R})$ and their derivatives belong to $L^2(\mathbb{R})$. Then, for any integer $\alpha \ge 0$ and $l\ge 0$, we have
\bma
	&\| \partial_x^\alpha (G_e^{11}\ast g)\|_{L^2} \le C(1+t)^{-\frac54-\frac{\alpha}{2}}\| g\|_{L^1\cap L^2} + C(1+t)^{-\frac{l+2}{2}}\| \partial_x^{\alpha+l}g\|_{L^2},\\
	&\| \partial_x^\alpha (G_e^{22}\ast g)\|_{L^2} \le C(1+t)^{-\frac54-\frac{\alpha}{2}}\| g\|_{L^1\cap L^2} + C(1+t)^{-\frac{l}{2}}\| \partial_x^{\alpha+l}g\|_{L^2} \label{Ge-22}\\
	&\| \partial_x^\alpha (G_e^{33}\ast g)\|_{L^2} \le C(1+t)^{-\frac14-\frac{\alpha}{2}}\| g\|_{L^1\cap L^2} + C(1+t)^{-\frac{l}{2}}\| \partial_x^{\alpha+l}g\|_{L^2},\\
	&\| \partial_x^\alpha (G_e^{12}\ast g)\|_{L^2},\| \partial_x^\alpha (G_e^{21}\ast g)\|_{L^2} \le  C(1+t)^{-\frac54-\frac{\alpha}{2}}\| g\|_{L^1\cap L^2} + C(1+t)^{-\frac{l+1}{2}}\| \partial_x^{\alpha+l}g\|_{L^2},\\
	&\| \partial_x^\alpha (G_e^{13}\ast g)\|_{L^2}, \| \partial_x^\alpha (G_e^{31}\ast g)\|_{L^2} \le  C(1+t)^{-\frac34-\frac{\alpha}{2}}\| g\|_{L^1\cap L^2} + C(1+t)^{-\frac{l+1}{2}}\| \partial_x^{\alpha+l}g\|_{L^2},\\
	&\| \partial_x^\alpha (G_e^{23}\ast g)\|_{L^2},\| \partial_x^\alpha (G_e^{32}\ast g)\|_{L^2} \le C(1+t)^{-\frac34-\frac{\alpha}{2}}\| g\|_{L^1\cap L^2} + C(1+t)^{-\frac{l}{2}}\| \partial_x^{\alpha+l}g\|_{L^2}.
\ema
\end{lem}
\begin{proof}
For simplicity, we only prove the estimate of $\partial_x^\alpha (G_e^{13}\ast g).$ By Plancherel's theorem, we obtain
\bmas
	\| \partial_x^\alpha (G_e^{13}\ast g)\|_{L^2}^2 &=
	\int_{|k|\le \epsilon}| (ik)^\alpha \hat{G}_e^{13}\hat{g}|^2\,dk + \int_{\epsilon\le|k|\le R}| (ik)^\alpha \hat{G}_e^{13}\hat{g}|^2\,dk + \int_{|k|\ge R}| (ik)^\alpha \hat{G}_e^{13}\hat{g}|^2\,dk\\
	&:=J_1 + J_2 +J_3.
\emas
Firstly, we estimate the term $J_3$. From Proposition \ref{upper_eEM},
\bmas
	J_3 &\le C\int_{|k|\ge R} |k|^{2\alpha}|k|^{-2} e^{-t}|\hat{g}|^2\,dk + C\int_{|k|\ge R} |k|^{2\alpha} |k|^{-2} e^{-\frac{t}{2k^2}}|\hat{g}|^2\,dk\\
	&\le Ce^{-t}\int_{|k|\ge R}|k|^{2(\alpha-1)}|\hat{g}|^2\,dk + C\sup_{|k| > R} \left( |k|^{-2l-2}e^{-\frac{t}{2k^2}}\right)\int_{|k|\ge R} |k|^{2(\alpha+l)}|\hat{g}|^2\,dk\\
	&\le  C(1+t)^{-(l+1)}\left\| \partial_x^{\alpha+l}g\right\|_{L^2}^2,
\emas
where we use the fact that
$$
\sup_{|k| > R} \left( |k|^{-l}e^{-\frac{t}{2k^2}}\right) \le C(1+t)^{-\frac{l}{2}},\quad \forall l\ge 0.
$$
Next, for $J_1$, we have
\bmas
	J_1 &\le C\int_{|k|\le\epsilon}|k|^{2(\alpha+1)} e^{-k^2t}|\hat{g}|^2\,dk + C\int_{|k|\le\epsilon} |k|^{2(\alpha+1)} e^{-\frac{t}{4}}|\hat{g}|^2\,dk \\
	&\le C\sup_{|k| \le \eps} |\hat{g} |^2\int_{|k|\le\epsilon}|k|^{2(\alpha+1)} e^{-k^2t}\,dk + Ce^{-\frac{t}{4}}\int_{|k|\le\epsilon} |\hat{g}|^2\,dk \\
&\le C(1+t)^{-\frac32-\alpha}\| g\|_{L^1\cap L^2}^2.
\emas
For $J_2$, we have
$$
	J_2 \le C\int_{\epsilon\le |k| \le R}|k|^{2\alpha} e^{-2ct}|\hat{g}|^2\,dk \le Ce^{-2ct}\| g\|_{L^2}^2 .
$$
Thus,
$$
	\| \partial_x^\alpha (G_e^{13}\ast g)\|_{L^2} \le C(1+t)^{-\frac34-\frac{\alpha}{2}}\| g\|_{L^1\cap L^2} + C(1+t)^{-\frac{l+1}{2}}\| \partial_x^{\alpha+l} g\|_{L^2}.
$$
The proof is completed.
\end{proof}

\begin{lem} \label{Ge-lowerbound}Let $\alpha\ge 0$ be an integer.
	 Assume that vector function $g=(g_1(x),g_2(x))$  with $g_1(x), g_2(x)$ belong to $
	 H^{2\alpha+3}(\mathbb{R})\cap L^1(\mathbb{R})$ and there exists a constant $d_0>0$ such that $\inf_{|k|\le \epsilon}|\hat{g}(k)|\ge d_0$. Then, for $t> 0$ being large enough, it holds
	\bma
	C_1(1+t)^{-\frac54-\frac{\alpha}{2}} &\le \| \partial_x^\alpha(G_e^{ij} \ast g)\|_{L^2} \le C_2(1+t)^{-\frac54-\frac{\alpha}{2}}, \quad i,j=1,2, \label{Ge-lb-1}\\
	C_1(1+t)^{-\frac34-\frac{\alpha}{2}} &\le \| \partial_x^\alpha(G_e^{i3} \ast g)\|_{L^2},~\| \partial_x^\alpha(G_e^{3i} \ast g)\|_{L^2} \le C_2(1+t)^{-\frac34-\frac{\alpha}{2}}, \quad i=1,2, \label{Ge-lb-2} \\
	C_1(1+t)^{-\frac14-\frac{\alpha}{2}} &\le \|\partial_x^\alpha(G_e^{33} \ast g)\|_{L^2} \le C_2(1+t)^{-\frac14-\frac{\alpha}{2}},\label{Ge-lb-3}
	\ema
	where  $C_2\ge C_1 >0$ are two constants.
\end{lem}

\begin{proof}

Firstly, from Lemma \ref{Ge-convolution}, if the the upper bounds of \eqref{Ge-lb-1}--\eqref{Ge-lb-3} hold, we need
$$
\begin{cases}
	g(x) \in H^{2\alpha+1}(\R)\cap L^1(\R), \quad i,j=1,3,\\
	g(x) \in H^{2\alpha+3}(\R)\cap L^1(\R), \quad i=j=2,\\
    g(x) \in H^{2\alpha+2}(\R)\cap L^1(\R),\quad other ~cases.
\end{cases}
$$

By Lemma \ref{Ge-convolution}, we only need to show the lower bounds. For $i,j=1,2,3,$ it holds that
\bmas
	\| \partial_x^\alpha(G_e^{ij}\ast g)\|_{L^2}^2 &=
	\int_{|k| \le \epsilon} |(ik)^\alpha\hat{G}_e^{ij}\hat{g}|^2\,dk +  \int_{ |k| > \epsilon} |(ik)^\alpha\hat{G}_e^{ij}\hat{g}|^2\,dk \\
	&\ge \int_{|k| \le \epsilon}\bigg[\frac{1}{2}|(ik)^\alpha(P_1^{ij}e^{\lambda_1t}+P_4^{ij}e^{\lambda_4t})\hat{g}|^2 - \epsilon^{2\alpha} \sum_{l=2,3,5,6}|  P_l^{ij}e^{\lambda_lt} \hat{g}|^2\bigg]\,dk \\
	&\ge \frac{1}{2}\int_{|k| \le \epsilon}k^{2\alpha}|(P_1^{ij}e^{\lambda_1t}+P_4^{ij}e^{\lambda_4t})\hat{g}|^2 \,dk- Ce^{-\frac{t}{4}},
\emas
where $P^{ij}_l$, $l=1,2,\cdots,6$ is the element of the matrix $P_l$.

It is easy to see that
\bmas
	P_1^{ij}e^{\lambda_1t}+P_4^{ij}e^{\lambda_4t} = (P_1^{ij}+P_4^{ij})e^{{\rm Re}\lambda_1t}\cos{({\rm Im}\lambda_1t)} + i(P_1^{ij}-P_4^{ij})e^{{\rm Re}\lambda_1t}\sin{({\rm Im}\lambda_1t)}.
\emas
For simplicity, we only estimate $\| G_e^{12}\ast g\|_{L^2}$. Recalling \eqref{Pj}--\eqref{Pl}, we have
\bmas
	&P_1^{12} = -\frac{\lambda_1}{2(\lambda_1-\lambda_2)(\lambda_1-\lambda_3)}(\mathbb{I}_2+i\mathbb{O}_1) :=p_1^{12}(\mathbb{I}_2+i\mathbb{O}_1),\\
	&P_4^{12} = -\frac{\lambda_4}{2(\lambda_4-\lambda_5)(\lambda_4-\lambda_6)}(\mathbb{I}_2-i\mathbb{O}_1) :=p_4^{12}(\mathbb{I}_2-i\mathbb{O}_1).
\emas
Since $\lambda_2+\lambda_3 = -(\lambda_1+1+i)$ and $\lambda_2\lambda_3 = -\lambda_1(\lambda_2+\lambda_3) + 1+k^2$, it follows that
$$
	(\lambda_1-\lambda_2)(\lambda_1-\lambda_3) = \lambda_1^2+2\lambda_1(\lambda_1+1+i)+1+k^2,
$$
Similarly,
$$
	(\lambda_4-\lambda_5)(\lambda_4-\lambda_6) = \lambda_4^2+2\lambda_4(\lambda_4+1-i)+1+k^2.
$$
For $|k|\le \epsilon$, from Lemma \ref{eEM-eigen-expansion}, we have
\bmas
	&p_1^{12} +p_4^{12} =-\frac{(1+k^2)(\lambda_1+\lambda_4) + |\lambda_1|^2[3(\lambda_1+\lambda_4)+4]}{2|\lambda_1-\lambda_2|^2|\lambda_1-\lambda_3|^2} = \frac{k^2 + O(k^4)}{|\lambda_1-\lambda_2|^2|\lambda_1-\lambda_3|^2} :=p_+^{12},\\
	&p_1^{12} -p_4^{12} =-\frac{(1+k^2)(\lambda_1-\lambda_4) + |\lambda_1|^2[3(\lambda_4-\lambda_1)-4i]}{2|\lambda_1-\lambda_2|^2|\lambda_1-\lambda_3|^2} = i\frac{k^2 + O(k^4)}{|\lambda_1-\lambda_2|^2|\lambda_1-\lambda_3|^2} :=ip_-^{12}.
\emas
Here, $p_{+}^{12}$ and $p_{-}^{12}$ are two real functions.
Therefore,
$$
	P_1^{12} + P_4^{12} = p_+^{12}\mathbb{I}_2 - p_-^{12}\mathbb{O}_1,\quad 	P_1^{12} - P_4^{12} =i( p_-^{12}\mathbb{I}_2 + p_+^{12}\mathbb{O}_1).
$$
Thus, we have
\bmas
	&\quad \int_{|k| \le \epsilon}k^{2\alpha} \left|\left( P_1^{12}e^{\lambda_1t} + P_4^{12}e^{\lambda_4t}\right)\hat{g}\right|^2\,dk\\
	&=\int_{|k| \le \epsilon}k^{2\alpha}\left| (p_+^{12}\mathbb{I}_2 - p_-^{12}\mathbb{O}_1)\hat{g}\cos({\rm Im}\lambda_1t) - ( p_-^{12}\mathbb{I}_2+p_+^{12}\mathbb{O}_1 )\hat{g}\sin({\rm Im}\lambda_1t) \right|^2e^{2{\rm Re}\lambda_1t}\,dk\\
	&=\int_{|k|\le \epsilon}k^{2\alpha}e^{2{\rm Re}\lambda_1t}\left[(p_+^{12})^2+(p_-^{12})^2\right]\left| \cos(\theta+{\rm Im}\lambda_1t)\hat{g} -\sin(\theta+{\rm Im}\lambda_1t)\mathbb{O}_1\hat{g}\right|^2\,dk,
\emas
where $\cos\theta= \frac{p_+^{12}}{\sqrt{(p_+^{12})^2+(p_-^{12})^2}}$ and $\sin\theta= \frac{p_-^{12}}{\sqrt{(p_+^{12})^2+(p_-^{12})^2}}$. Since
\bmas
	\cos(\theta+{\rm Im}\lambda_1t)\hat{g} - \sin(\theta+{\rm Im}\lambda_1t)\mathbb{O}_1\hat{g} = \left( \ba \cos(\theta+{\rm Im}\lambda_1t) & \sin(\theta+{\rm Im}\lambda_1t) \\ -\sin(\theta+{\rm Im}\lambda_1t) & \cos(\theta+{\rm Im}\lambda_1t)  \ea \right)
	\left( \ba \hat{g}_1 \\ \hat{g}_2 \ea \right)
\emas
is a rotation transform of the vector $\hat{g}=(\hat{g}_1,\hat{g}_2)^{T}$, it follows that
\bmas
	&\quad \int_{|k| \le \epsilon} k^{2\alpha}\left|( P_1^{12}e^{\lambda_1t} + P_4^{12}e^{\lambda_4t})\hat{g}\right|^2\,dk \\
	&= \int_{|k|\le \epsilon} k^{2\alpha}e^{2{\rm Re}\lambda_1t}\left[(p_+^{12})^2 + (p_-^{12})^2\right]|\hat{g}|^2\,dk\\
	&\ge \int_{|k|\le \epsilon} \frac{k^{2\alpha}e^{2{\rm Re}\lambda_1t}}{|\lambda_1-\lambda_2|^4|\lambda_1-\lambda_3|^4}\left[ k^4 -O(k^6) \right] |\hat{g}|^2\,dk\\
	&\ge C'\int_{|k|\le \epsilon} e^{-3k^2t}k^{4+2\alpha}|\hat{g}|^2\,dk - C''\int_{|k|\le \epsilon} e^{-\frac{k^2}{2}t}k^{6+2\alpha}|\hat{g}|^2\,dk ,
\emas
where we use the following estimate
\bmas
	&|\lambda_1-\lambda_j| \ge s_0, \quad j=2,3,~~|k| \le \epsilon,\\
	&{\rm Re}\lambda_1 = -k^2(1+O(k^2)) \ge -\frac32k^2, \quad |k|\le \epsilon.
\emas
And it holds for $t\ge t_0 =\frac{L^2}{\epsilon^2}$ with the constant $L > 1$ that
\bmas
	\int_{|k|\le \epsilon} e^{-3k^2t}k^{4+2\alpha}|\hat{g}|^2\,dk&\ge d_0^2t^{-\frac{5}{2}-\alpha}\int_0^{\epsilon t^{1/2}} r^{4+2\alpha}e^{-3r^2}\,dr \\
&\ge d_0^2t^{-\frac{5}{2}-\alpha}\int_0^L r^{4+2\alpha}e^{-3r^2} \,dr \ge \frac{ d_0^2L^{5+2\alpha}e^{-3L^2}}{5+2\alpha}t^{-\frac{5}{2}-\alpha}.
\emas
Summarizing above, for $t >0$ being large enough, we obtain that
$$
	\| G_e^{12} \ast g \|_{L^2}^2 \ge C'''(1+t)^{-\frac{5}{2}-\alpha} - C''(1+t)^{-\frac{7}{2}-\alpha} - Ce^{-\frac{t}{4}} \ge C_1(1+t)^{-\frac{5}{2}-\alpha}.
$$
The other terms can be proved similarly, and we omit the details. Then we complete this proof.
\end{proof}

\begin{cor} \label{Ur_linear_time}
	Suppose $U_r = (u_{r},E_{r},B_{r})$ is the solution of Cauchy problem \eqref{eEM}. 	Then, for any integer $\alpha \ge 0$ and $l\ge 0$, we have the following time decay rates:
	\bmas
	\| \partial_x^\alpha u_r\|_{L^2} &\le C(1+t)^{-\frac34-\frac{\alpha}{2}}\| (u_{r0},E_{r0},B_{r0})\|_{L^1\cap L^2} + C(1+t)^{-\frac{l+1}{2}}\| \partial_x^{\alpha+l}(u_{r0},E_{r0},B_{r0})\|_{L^2},\\ 
	\| \partial_x^\alpha E_r\|_{L^2} &\le C(1+t)^{-\frac34-\frac{\alpha}{2}}\| (u_{r0},E_{r0},B_{r0})\|_{L^1\cap L^2} + C(1+t)^{-\frac{l}{2}}\| \partial_x^{\alpha+l}(u_{r0},E_{r0},B_{r0})\|_{L^2},\\
	\| \partial_x^\alpha B_r\|_{L^2} &\le C(1+t)^{-\frac14-\frac{\alpha}{2}}\| (u_{r0},E_{r0},B_{r0})\|_{L^1\cap L^2} + C(1+t)^{-\frac{l}{2}}\| \partial_x^{\alpha+l}(u_{r0},E_{r0},B_{r0})\|_{L^2},
	\emas
where $C$ is a positive constant.
 Moreover, if $u_{r0}\in H^{2\alpha+2}(\R)\cap L^1(\R)$ and $E_{r0}\in H^{2\alpha+3}(\R)\cap L^1(\R)$, there exists a constant $d_0>0$ such that $\inf_{|k|\le \epsilon}|\hat{B}_{r0}(k)| \ge d_0$, then, it holds that
$$
	\| \dxa u_r\|_{L^2} \ge C_1(1+t)^{-\frac34-\frac{\alpha}{2}},~\| \dxa E_r\|_{L^2} \ge C_1(1+t)^{-\frac34-\frac{\alpha}{2}},~\| \dxa B_r\|_{L^2} \ge C_1(1+t)^{-\frac14-\frac{\alpha}{2}},
$$
for $t>0$ is sufficiently large and $C_1$ is a positive constant.
\end{cor}

\begin{proof}
	From $U_r = G_e\ast U_r(0)$ and Lemma \ref{Ge-convolution}, it is easy to obtain the upper bounds on the time decay rates of solutions. For the lower bounds, from Lemma \ref{Ge-lowerbound}, we have
	\bmas
		\| \partial_x^\alpha u_r \|_{L^2} &= \| \dxa G_e^{11}\ast u_{r0} +\dxa  G_e^{12}\ast E_{r0} + \dxa G_e^{13}\ast B_{r0}\|_{L^2}\\
		&\ge \| \dxa G_e^{13}\ast B_{r0}\|_{L^2} - \|\dxa  G_e^{11}\ast u_{r0}\|_{L^2} -\| \dxa G_e^{12}\ast E_{r0}\|_{L^2}\\
		&\ge C_1d_0(1+t)^{-\frac34-\frac{\alpha}{2}} - C_2(1+t)^{-\frac54-\frac{\alpha}{2}}(\| u_{r0}\|_{H^{2\alpha+1}\cap L^1} + \| E_{r0}\|_{H^{2\alpha+2}\cap L^1})\\
		&\ge C_1(1+t)^{-\frac34-\frac{\alpha}{2}}
	\emas
for $t>0$ is sufficiently large. The proof is similar for $E_r$ and $B_r$.
\end{proof}

\setcounter{equation}{0}			
\section{The Original Nonlinear Problem}
			
\subsection{Energy Estimate}
			
In this section, we rewrite the system \eqref{re-em_1} for $W=(\rho,u_{1},u_{r},E_{1},E_{r},B_{r})$ as following:
\bma \label{em-2}
\begin{cases}
	\partial_t \rho + (\rho+1)\partial_x u_1 + u_1\partial_x \rho = 0, \\
	\partial_t u_1  + \frac{P'(\rho +1)}{\rho+1}\partial_x \rho+   E_1 +   u_1 =  - u_1 \partial_x u_1 +  u_r\cdot\mathbb{O}_1 B_r, \\
	\partial_t u_r -\mathbb{O}_1u_r +   u_r +   E_r =  - u_1 \partial_x u_r -  u_1\mathbb{O}_1 B_r,\\
	\partial_tE_1 = (\rho+1)u_1,\\
	\partial_t E_r - \mathbb{O}_1\partial_x B_r =   (\rho+1)u_r, \\
	\partial_t B_r + \mathbb{O}_1\partial_x E_r = 0, \\
	\partial_x E_1 = - \rho,
\end{cases}
\ema
with initial data
\be \label{init}
W|_{t=0}=W(0):=(\rho_0,u_{10},u_{r0},E_{10},E_{r0},B_{r0})^{T},
\ee
satisfying
\be \label{init-compatible}	\partial_x   E_{10}=- \rho_0.\ee
In what follows, we assume that
\be \label{assumption} \sup_{0\le s\le t}\| W (s)\|_{H^N} \le \delta \quad {\rm for }\quad t\in[0,T],\ee
 where $\delta$ is a small positive constant. Under this assumption, we can see that $n=\rho+1\ge const. >0$ and $\left| P'(\rho+1)\right|$ is bounded.
			
For later use, we list some Sobolev inequalities as follows.

\begin{lem}[\cite{H-K}] \label{sobolev-ineq}
Let $1\le p,r,q\le \infty$, and $\frac1p =\frac1r+\frac1q$. Then, we have
\begin{enumerate}
\item[(1)] For $\alpha \ge 0$:
\be
	\| \partial_x^\alpha(uv)\|_{L^p} \le C\| u \|_{L^q}\| \partial_x^\alpha v\|_{L^r} +C\| v\|_{L^q}\| \partial^\alpha u\|_{L^r}.
\ee
\item[(2)] For $\alpha \ge 1$:
\be
	\|\partial_x^\alpha(uv_x)-u\partial_x^\alpha v_x\|_{L^p} \le C\| v_x \|_{L^q}\| \partial_x^\alpha u\|_{L^r} +C\| u_x\|_{L^q}\| \partial_x^\alpha v\|_{L^r}.
\ee
\end{enumerate}
\end{lem}

Now, we give the energy estimates for the solutions of system \eqref{em-2}.
\begin{lem} \label{zero-energy-lem}
Under the assumption \eqref{assumption}, it holds
\be \label{zero-energy}
\|W(t)\|_{L^2}^2  + \intt\| (u_1,u_r)\|_{L^2}^2 d\tau \le C\| W(0) \|_{L^2}^2 + C\intt M_1\| (\rho,u_1,u_r)\|_{L^2}^2d\tau,
\ee
where  $C$ is a positive constant and $M_1 =\| \partial_{x} (\rho,u_1) \|_{L^\infty} +  \| u_1 \|_{L^\infty}\| \partial_x\rho \|_{L^\infty} + \| \rho \|_{L^\infty}\| \partial_x u_1 \|_{L^\infty} $.
\end{lem}

\begin{proof}
We  take the inner product of the first equation of \eqref{em-2} with $\frac{P'(\rho+1)}{\rho+1}\rho$, one has
\be \label{rho-zero}
	\frac{1}{2}\int_{\mathbb{R}}\frac{P'(\rho+1)}{\rho+1}\partial_t\rho^2 \,dx +  \int_{\mathbb{R}} P'(\rho+1)\rho\partial_x u_1\,dx +  \frac{1}{2}\int_{\mathbb{R}} \frac{P'(\rho+1)}{\rho+1}u_1\partial_x (\rho^2) \,dx  =0.
\ee
Next, taking the inner product of the second equation of \eqref{em-2} with $(\rho+1)u_1$, we have
\bma \label{u1-zero}
	&\quad\frac{1}{2}\int_{\mathbb{R}}(\rho+1)\partial_t u_1^2 \,dx +  \int_{\mathbb{R}} (\rho+1)u_1E_1\,dx +  \int_{\mathbb{R}}(\rho+1)u_1^2 \,dx\nnm\\
	& =-  \frac{1}{2}\int_{\mathbb{R}}(\rho+1)u_1\partial_x(u_1^2) \,dx -  \int_{\mathbb{R}}P'(\rho+1)u_1\partial_x\rho  \,dx +  \int_{\mathbb{R}} (\rho+1)u_1u_r\cdot\mathbb{O}_1B_r\,dx.	
\ema
Taking the inner product of third equation of \eqref{em-2} with $(\rho+1)u_r$, we obtain
\bma \label{ur-zero}
	&\quad\frac{1}{2}\int_{\mathbb{R}}(\rho+1)\partial_t |u_r|^2 \,dx + \int_{\mathbb{R}} (\rho+1)u_r\cdot E_r\,dx +  \int_{\mathbb{R}}(\rho+1)|u_r|^2 \,dx \nnm\\
	& =- \frac{1}{2}\int_{\mathbb{R}}(\rho+1)u_1\partial_x |u_r|^2 \,dx-  \int_{\mathbb{R}} (\rho+1)u_1\mathbb{O}_1B_r\cdot u_r\,dx.
\ema
Here, we have used the fact that $\mathbb{O}_1u_r\cdot u_r=0.$
Taking the inner product of fourth, fifth and sixth equations of \eqref{em-2} with $E_1,~E_r$ and $B_r$ respectively, one has
\bma \label{E-B-zero}
	&\quad\frac{1}{2}\int_{\mathbb{R}}\partial_tE_1^2\,dx + \frac{1}{2}\int_{\mathbb{R}}\partial_t|E_r|^2\,dx +\frac{1}{2}\int_{\mathbb{R}}\partial_t|B_r|^2\,dx \nnm\\
	&=  \int_{\mathbb{R}} (\rho+1)u_1E_1\,dx+  \int_{\mathbb{R}} (\rho+1)u_r\cdot E_r\,dx+\int_{\mathbb{R}} \mathbb{O}_1\partial_xB_r\cdot E_r\,dx-\int_{\mathbb{R}} \mathbb{O}_1\partial_xE_r\cdot B_r\,dx.
\ema	
Adding \eqref{rho-zero}--\eqref{E-B-zero} together, we obtain
\bma \label{dt-zero-enrgy}
	&\frac{d}{2dt}\left(  \int_{\mathbb{R}}\frac{P'(\rho+1)}{\rho+1}\rho^2 \,dx + \int_{\mathbb{R}}(\rho+1)(u_1^2+|u_r|^2) \,dx +\|( E_1,E_r,B_r ) \|_{L^2}^2 \right) \nnm\\
	&+ \int_{\mathbb{R}}(\rho+1)(u_1^2 +|u_r|^2) \,dx =I_1 +I_2,
\ema
where
\bmas
	&I_1 = \frac{1}{2}\int_{\mathbb{R}}\partial_t\left(\frac{P'(\rho+1)}{\rho+1}\right)\rho^2 \,dx +\frac{1}{2}\int_{\mathbb{R}}\partial_t(\rho+1) (u_1^2 +|u_r|^2 ) \,dx,\\
	&I_2 = -\int_{\mathbb{R}} P'(\rho+1)\partial_x(\rho u_1) \,dx-\frac{1}{2}\int_{\mathbb{R}} \frac{P'(\rho+1)}{\rho+1}u_1\partial_x \rho^2\,dx -\frac{1}{2}\int_{\mathbb{R}}(\rho+1)u_1\partial_x (u_1^2+|u_r|^2 )\,dx.
\emas
Under the assumption \eqref{assumption}, we have
\bmas
	\left\|\partial_t\left(\frac{P'(\rho+1)}{\rho+1}\right)\right\|_{L^\infty}  &\le C\| \partial_t \rho\|_{L^\infty}=\| \partial_x u_1+ \partial_x (\rho u_1)\|_{L^\infty}\\
	&\le C\left( \| \partial_x u_1 \|_{L^\infty} +\| u_1\|_{L^\infty}\| \partial_x \rho\|_{L^\infty} + \|  \rho\|_{L^\infty}\| \partial_x u_1\|_{L^\infty}\right).
\emas
And then, for $I_1$, we have
$$
	|I_1| \le C\left( \| \partial_x u_1 \|_{L^\infty} +\| u_1\|_{L^\infty}\| \partial_x \rho\|_{L^\infty}+ \|  \rho\|_{L^\infty}\| \partial_x u_1\|_{L^\infty} \right) \| (\rho,u_1,u_r) \|_{L^2}^2.
$$
For $I_2$, by integrating by parts, we have
\bmas
	|I_2| &\le C\| \partial_x(P'(\rho+1)) \|_{L^\infty}\| \rho u_1 \|_{L^1} + C\left\| \partial_x\left( \frac{P'(\rho+1)}{\rho+1}u_1\right) \right\|_{L^\infty}\| \rho \|_{L^2}^2 \\
	&\quad  + C\left\| \partial_x\left[(\rho+1)u_1\right] \right\|_{L^\infty}(\| u_1 \|_{L^2}^2 +\| u_r \|_{L^2}^2 )\\
	&\le C ( \| \partial_x(\rho,u_1)\|_{L^\infty} +\| u_1\|_{L^\infty}\| \partial_x \rho\|_{L^\infty}  + \| \rho\|_{L^\infty}\| \partial_x u_1\|_{L^\infty} ) \| (\rho,u_1,u_r) \|_{L^2}^2.
\emas
We see that $\int_{\mathbb{R}}\frac{P'(\rho+1)}{\rho+1}\rho^2 dx +\int_{\mathbb{R}}(\rho+1) (u_1^2+|u_r|^2) dx  $ is equivalent to $\| (\rho,u_1,u_r)\|_{L^2}^2$. Under the assumption \eqref{assumption}, further integrating  \eqref{dt-zero-enrgy} with respect to $t$, one has
$$
	\| W(t)\|_{L^2}^2 + \int_0^t\| (u_1,u_r)\|_{L^2}^2 \,d\tau\le C\|W(0)\|_{L^2}^2+  C\int_0^t M_1\| (\rho,u_1,u_r)\|_{L^2}^2 \,d\tau
$$
with $M_1 = \| u_1 \|_{L^\infty}\| \partial_x\rho \|_{L^\infty} + \| \rho \|_{L^\infty}\| \partial_x u_1 \|_{L^\infty} + \| \partial_{x} (\rho,u_1) \|_{L^\infty}.$
Thus, we complete the proof.
\end{proof}

Next, we show the energy estimates for derivatives of solution.
\begin{lem} \label{alpha-energy-lem}
Under the assumption \eqref{assumption},  it holds for $1 \le \alpha \le N$ that
\bma \label{alpha-energy}
	&\quad \| \partial_x^\alpha W(t)\|_{L^2}^2+ \int _0^t\| \partial_x^\alpha (u_1, u_r) \|_{L^2}^2d\tau \nnm\\
	& \le C\|\partial_x^\alpha W(0)\|_{L^2}^2 + C\int _0^t M_2\| \partial_x^\alpha(\rho,u_1,u_r)\|_{L^2}^2  +\| B_r\|_{L^\infty}\| \partial_x^\alpha(u_1,u_r)\|_{L^2}^2d\tau \nnm\\
	& \quad   +C\int _0^t\| \partial_x^\alpha(E_1,E_r,B_r)\|_{L^2}\| (\partial_x \rho,u_1,u_r)\|_{H^1} \| \partial_x^\alpha(\rho,u_1,u_r)\|_{L^2} \,d\tau,
\ema
where $M_2=\| \partial_x (\rho,u_1,u_r)\|_{L^\infty}+ \| \partial_x \rho\|_{L^\infty}\| u_1\|_{L^\infty} +\| \partial_x u_1\|_{L^\infty}\| \rho\|_{L^\infty}$.
\end{lem}
			
\begin{proof}
We apply $\partial_x^\alpha$ to the first equation of \eqref{em-2} and take the inner product with $\frac{P'(\rho+1)}{\rho+1}\partial_x^\alpha\rho$, we have
\bma
	&\quad\frac{1}{2}\int_{\mathbb{R}}\frac{P'(\rho+1)}{\rho+1}\partial_t(\partial_x^\alpha\rho)^2\,dx \nnm \\
	&= - \int_{\mathbb{R}} \frac{P'(\rho+1)}{\rho+1} \partial_x^\alpha\rho \partial_x^\alpha(u_1\partial_x\rho)  \,dx - \int_{\mathbb{R}} \frac{P'(\rho+1)}{\rho+1}\partial_x^\alpha\rho  \partial_x^\alpha[(\rho+1)\partial_x u_1] \,dx. \label{rho-high}
\ema

Applying $\partial_x^\alpha$ to the second equation of \eqref{em-2} and taking the inner product with $(\rho+1)\partial_x^\alpha u_1$, we obtain
\bma
	&\quad\frac{1}{2}\int_{\mathbb{R}}(\rho+1)\partial_t(\partial_x^\alpha u_1)^2\,dx +  \int_{\mathbb{R}} (\rho+1)\partial_x^\alpha u_1\partial_x^\alpha E_1 \,dx + \int_{\mathbb{R}}(\rho+1)(\partial_x^\alpha u_1)^2 \,dx \nnm\\
	&= - \int_{\mathbb{R}}(\rho+1) \partial_x^\alpha u_1 \partial_x^\alpha(u_1\partial_x u_1) \,dx -\int_{\mathbb{R}} (\rho+1) \partial_x^\alpha u_1 \partial_x^\alpha\left(\frac{P'(\rho+1)}{\rho+1}\partial_x\rho\right) \,dx \nnm\\
	&\quad +  \int_{\mathbb{R}} (\rho+1)\partial_x^\alpha u_1\partial_x^\alpha(u_r\cdot\mathbb{O}_1B_r)\,dx.
\ema

Applying $\partial_x^\alpha$ to the third equation of \eqref{em-2} and taking the inner product with $(\rho+1)\partial_x^\alpha u_r$, it yields
\bma
	&\quad\frac{1}{2} \int_{\mathbb{R}}(\rho+1)\partial_t |\partial_x^\alpha u_r|^2 \,dx + \int_{\mathbb{R}}(\rho+1)|\partial_x^\alpha u_r|^2 \,dx +  \int_{\mathbb{R}} (\rho+1)\partial_x^\alpha u_r \cdot\partial_x^\alpha E_r \,dx\nnm\\
	&= - \int_{\mathbb{R}} (\rho+1)\partial_x^\alpha u_r \cdot  \partial_x^\alpha(u_1\partial_x u_r)  \,dx  -\int_{\mathbb{R}} (\rho+1) \partial_x^\alpha u_r \cdot \partial_x^\alpha(u_1\mathbb{O}_1B_r) \,dx.
\ema

Next, applying $\partial_x^\alpha$ to the fourth, fifth and sixth equations of \eqref{em-2} and taking the inner product with $\partial_x^\alpha E_1,~\partial_x^\alpha E_r$ and $\partial_x^\alpha B_r$ respectively, we obtain
\bma \label{E-B-high}
	&\quad \frac{d}{2dt}\| \partial_x^\alpha E_1\|_{L^2}^2 + \frac{d}{2dt}\| \partial_x^\alpha E_r\|_{L^2}^2 + \frac{d}{2dt}\| \partial_x^\alpha B_r\|_{L^2}^2 \nnm\\
	&=  \int_{\mathbb{R}} \partial_x^\alpha[(\rho+1)u_1] \partial_x^\alpha E_1\,dx +  \int_{\mathbb{R}} \partial_x^\alpha[(\rho+1)u_r] \cdot \partial_x^\alpha E_r\,dx\nnm\\
	&\quad +\int_{\mathbb{R}} \mathbb{O}_1\partial_x^{\alpha+1}B_r \cdot \partial_x^\alpha E_r\,dx -\int_{\mathbb{R}} \mathbb{O}_1\partial_x^{\alpha+1}E_r \cdot \partial_x^\alpha B_r\,dx.
\ema
Adding \eqref{rho-high}--\eqref{E-B-high} together, we have
\bma \label{dt-alpha-energy}
	&\frac{d}{2dt}\left(  \int_{\mathbb{R}}\frac{P'(\rho+1)}{\rho+1}(\partial_x^\alpha \rho)^2  \,dx + \int_{\mathbb{R}}(\rho+1) |\partial_x^\alpha u |^2 \,dx + \| \partial_x^\alpha ( E ,B_r ) \|_{L^2}^2 \right)\nnm \\
	& + \int_{\mathbb{R}}(\rho+1) |\partial_x^\alpha u |^2 \,dx =g_1+g_2+g_3+g_4,
\ema
where $u=(u_1,u_r)$, $E=(E_1,E_r)$, and
\bmas
	g_1 & =\frac12 \int_{\mathbb{R}}\partial_t\left(\frac{P'(\rho+1)}{\rho+1}\right)(\partial_x^\alpha\rho)^2\,dx +  \frac{1}{2}\int_{\mathbb{R}}\partial_t(\rho+1) |\partial_x^\alpha u |^2  \,dx,\\
	g_2 &= -\frac{1}{2}\int_{\mathbb{R}} \frac{P'(\rho+1)}{\rho+1}u_1\partial_x( \partial_x^\alpha\rho)^2\,dx - \int_{\mathbb{R}} P'(\rho+1)\partial_x(\partial_x^\alpha \rho \partial_x^\alpha u_1) \,dx\\
	&\quad -\frac{1}{2}\int_{\mathbb{R}} (\rho+1)u_1\partial_x  |\partial_x^\alpha u|^2 \,dx ,\\
	g_3 & = -\int_{\mathbb{R}} \frac{P'(\rho+1)}{\rho+1}\left[ \partial_x^\alpha(u_1\partial_x\rho)-u_1\partial_x^{\alpha+1}\rho \right] \partial_x^\alpha\rho \,dx\\
	&\quad -\int_{\mathbb{R}} \frac{P'(\rho+1)}{\rho+1}\left[ \partial_x^\alpha[(\rho+1)\partial_x u_1]-(\rho+1)\partial_x^{\alpha+1} u_1 \right] \partial_x^\alpha\rho \,dx\\
	&\quad-\int_{\mathbb{R}}(\rho+1)\left[\partial_x^\alpha(u_1\partial_x u)-u_1\partial_x^{\alpha+1}u\right]\cdot \partial_x^\alpha u \,dx\\
	&\quad-\int_{\mathbb{R}} (\rho+1)\left[\partial_x^\alpha\left(\frac{P'(\rho+1)}{\rho+1}\partial_x\rho\right)-\frac{P'(\rho+1)}{\rho+1}\partial_x^{\alpha+1}\rho\right] \partial_x^\alpha u_1 \,dx, \\
	g_4 &=  \int_{\mathbb{R}} \partial_x^\alpha[(\rho+1)u]\cdot\partial_x^\alpha E \,dx - \int_{\mathbb{R}} (\rho+1)\partial_x^\alpha E\cdot \partial_x^\alpha u \,dx \\
	&\quad +\int_{\mathbb{R}} (\rho+1)\partial_x^\alpha(u_r\cdot\mathbb{O}_1B_r)\partial_x^\alpha u_1\,dx -\int_{\mathbb{R}} (\rho+1)\partial_x^\alpha(u_1\mathbb{O}_1B_r) \cdot \partial_x^\alpha u_r\,dx.
\emas
For $g_1$, 
we have
\bmas
	|g_1| &\le C\left\|\partial_t\left(\frac{P'(\rho+1)}{\rho+1}\right)\right\|_{L^\infty}\| \partial_x^\alpha \rho \|_{L^2}^2 + C\left\|\partial_t\left(\rho+1\right)\right\|_{L^\infty}\left(\| \partial_x^\alpha u_1 \|_{L^2}^2 +\| \partial_x^\alpha u_r \|_{L^2}^2\right) \\
	&\le CM_1\left( \| \partial_x^\alpha \rho \|_{L^2}^2 +\| \partial_x^\alpha u_1 \|_{L^2}^2 +\| \partial_x^\alpha u_r \|_{L^2}^2\right).
\emas
Here, $M_1$ is defined in Lemma \ref{zero-energy-lem}. For $g_2$, using integral by parts, we have
\bmas
	|g_2| &\le C \left\| \partial_x\left(\frac{P'(\rho+1)}{\rho+1}u_1\right)\right\|_{L^\infty} \| \partial_x^\alpha\rho\|_{L^2}^2 + C\| \partial_x(P'(\rho+1))\|_{L^\infty}(\| \partial_x^\alpha\rho\|_{L^2}^2 + \| \partial_x^\alpha u_1\|_{L^2}^2) \\
	&\quad +C\| \partial_x[(\rho+1)u_1]\|_{L^\infty}(\| \partial_x^\alpha u_1\|_{L^2}^2 + \| \partial_x^\alpha u_r\|_{L^2}^2)\\
	&\le CM_1\left(\| \partial_x^\alpha \rho\|_{L^2}^2+\| \partial_x^\alpha u_1\|_{L^2}^2 + \| \partial_x^\alpha u_r\|_{L^2}^2\right).
\emas
For $g_3$, by the second estimate in Lemma \ref{sobolev-ineq}, we have
\bmas
	&\quad \int_{\mathbb{R}} \frac{P'(\rho+1)}{\rho+1}[\partial_x^\alpha(u_1\partial_x \rho)-u_1\partial_x^{\alpha+1}\rho]\partial_x^\alpha \rho\,dx\\
	&\le\left\| \frac{P'(\rho+1)}{\rho+1}\right\|_{L^\infty}\left\| \partial_x^\alpha(u_1\partial_x \rho)-u_1\partial_x^{\alpha+1}\rho\right\|_{L^2}\| \partial_x^\alpha\rho\|_{L^2}\\
	&\le C\left( \| \partial_x\rho \|_{L^\infty}\| \partial_x^\alpha u_1\|_{L^2} +\| \partial_xu_1 \|_{L^\infty}\| \partial_x^\alpha \rho\|_{L^2}\right)\| \partial_x^\alpha\rho\|_{L^2}.
\emas
Similarly, we have
\bmas
	&\quad \int_{\mathbb{R}} \frac{P'(\rho+1)}{\rho+1}\left\{ \partial_x^\alpha[(\rho+1)\partial_x u_1]-(\rho+1)\partial_x^{\alpha+1} u_1 \right\} \partial_x^\alpha\rho \,dx\\
	&\le C\left(\| \partial_x \rho \|_{L^\infty}\| \partial_x^\alpha u_1\|_{L^2} + \| \partial_x u_1\|_{L^\infty}\| \partial_x^\alpha \rho\|_{L^2}\right)\| \partial_x^\alpha \rho\|_{L^2}.
\emas
And under the assumption \eqref{assumption}, we have
\bmas
	&\quad \int_{\mathbb{R}} (\rho+1)\left[\partial_x^\alpha\left(\frac{P'(\rho+1)}{\rho+1}\partial_x\rho\right)-\frac{P'(\rho+1)}{\rho+1}\partial_x^{\alpha+1}\rho\right] \partial_x^\alpha u_1 \,dx \\
&\le C\| \partial_x\rho\|_{L^\infty}\| \partial_x^\alpha \rho\|_{L^2} \| \partial_x^\alpha u_1\|_{L^2},\\
	&\quad \int_{\mathbb{R}}(\rho+1)\left[\partial_x^\alpha(u_1\partial_x u)-u_1\partial_x^{\alpha+1}u\right] \partial_x^\alpha u \,dx \\
	&\le C\left(\| \partial_x u_1\|_{L^\infty}\| \partial_x^\alpha u\|_{L^2} + \| \partial_x u\|_{L^\infty}\| \partial_x^\alpha u_1\|_{L^2}\right)\| \partial_x^\alpha u\|_{L^2}.
\emas
Thus, we obtain that
\bmas
	|g_3| \le C\| \partial_x(\rho,u_1,u_r)\|_{L^\infty}\left(\| \partial_x^\alpha \rho\|_{L^2}^2+\| \partial_x^\alpha u_1\|_{L^2}^2 + \| \partial_x^\alpha u_r\|_{L^2}^2\right).
\emas
For $g_4$, we have the following estimates
\bmas
	&\quad\int_{\mathbb{R}} \partial_x^\alpha[(\rho+1)u]\cdot\partial_x^\alpha E \,dx - \int_{\mathbb{R}} (\rho+1)\partial_x^\alpha E \cdot\partial_x^\alpha u \,dx\\
	&= \int_{\mathbb{R}} (\partial_x^\alpha[(\rho+1)u]- (\rho+1)\partial_x^\alpha u)\cdot \partial_x^\alpha E \,dx \\
	&\le C \left( \| u\|_{L^\infty}\| \partial_x^\alpha \rho\|_{L^2}+ \| \partial_x \rho\|_{L^\infty}\| \partial_x^{\alpha-1} u\|_{L^2}\right)\| \partial_x^\alpha E\|_{L^2}.
\emas
And under the assumption \eqref{assumption}, we have
\bmas
	&\int_{\mathbb{R}} (\rho+1)\partial_x^\alpha(u_r\cdot\mathbb{O}_1B_r)\partial_x^\alpha u_1\,dx \le C\left( \| u_r\|_{L^\infty}\| \partial_x^\alpha B_r\|_{L^2} + \| B_r\|_{L^\infty}\| \partial_x^\alpha u_r\|_{L^2}\right)\| \partial_x^\alpha u_1\|_{L^2}, \\
	&\int_{\mathbb{R}} (\rho+1)\partial_x^\alpha(u_1\mathbb{O}_1B_r)\partial_x^\alpha u_r\,dx\le C\left( \| u_1\|_{L^\infty}\| \partial_x^\alpha B_r\|_{L^2} + \| B_r\|_{L^\infty}\| \partial_x^\alpha u_1\|_{L^2}\right)\| \partial_x^\alpha u_r\|_{L^2}.
\emas
Therefore, we have
\bmas
	|g_4| &\le C \left( \| (u_1,u_r)\|_{L^\infty} + \| \partial_x\rho\|_{L^\infty}\right)\left( \| \partial_x^\alpha\rho\|_{L^2}+\| \partial_x^{\alpha-1}(u_1,u_r)\|_{L^2}\right)\| \partial_x^\alpha (E_1,E_r)\|_{L^2}\\
	&\quad+C\| (u_1,u_r)\|_{L^\infty}\| \partial_x^\alpha B_r\|_{L^2}\| \partial_x^\alpha(u_1,u_r)\|_{L^2}+C\| B_r\|_{L^\infty}\| \partial_x^\alpha(u_1,u_r)\|_{L^2}^2.
\emas
Under the assumption \eqref{assumption} and by integrating \eqref{dt-alpha-energy} with respect to $t$, we obtain  the energy estimate \eqref{alpha-energy}. Thus, we complete the proof of Lemma \ref{alpha-energy-lem}.
\end{proof}
			
\begin{lem}
It holds that
\bma \label{dissipation-rho-E}
\begin{split}
	&\quad \sum_{ \alpha=0}^{N-1}\frac{d}{dt}  \int_{\mathbb{R}} (\partial_x^\alpha u_1\partial_x^{\alpha+1}\rho+\partial_x^\alpha u_1\partial_x^{\alpha}E_1 + \partial_x^\alpha u_r\cdot\partial_x^{\alpha}E_r)\,dx  +  (\| \rho\|_{H^{N}}^2 + \| (E_1,E_r)\|_{H^{N-1}}^2) \\
	&\le \frac{C}{\epsilon}  \| (u_1,u_r)\|_{H^N}^2  + \epsilon \| \partial_x B_r\|_{H^{N-2}}^2  \\
	&\quad +C \| (\rho, u_1, u_r, B_r)\|_{H^2} ( \| (\rho,u_1,u_r)\|_{H^N}^2 + \| (E_1,E_r) \|_{H^{N-1}}^2 + \| \partial_x B_r\|_{H^{N-2}}^2 ),
\end{split}
\ema
where $\epsilon$ and $C$ are positive constants.
\end{lem}
			
\begin{proof}
For $0\le \alpha \le N-1$, applying $\partial_x^\alpha$ to the second equation of \eqref{em-2} and taking the inner product with $\partial_x^{\alpha+1}\rho$. Then, we have
\bmas
	&\quad \int_{\mathbb{R}} \partial_t(\partial_x^\alpha u_1)\partial_x^{\alpha+1}\rho\,dx + \int_{\mathbb{R}}\frac{P'(\rho+1)}{\rho +1}\left( \partial_x^{\alpha+1}\rho\right)^2 \,dx +\int_{\mathbb{R}}   \partial_x^\alpha E_1\partial_x^{\alpha+1}\rho\,dx\\
	&= -\int_{\mathbb{R}}  \partial_x^\alpha u_1\partial_x^{\alpha+1}\rho\,dx - \int_{\mathbb{R}} f\partial_x^{\alpha+1}\rho\,dx,
\emas
with
\bmas
	&f = \partial_x^\alpha(u_1\partial_x u_1) - \partial_x^\alpha(u_r\cdot\mathbb{O}_1B_r) + \left[ \partial_x^\alpha\left( \frac{P'(\rho+1)}{\rho+1}\partial_x\rho \right) - \frac{P'(\rho+1)}{\rho+1}\partial_x^{\alpha+1}\rho\right].
\emas
Notice that
\bmas
	&\partial_t(\partial_x^\alpha u_1)\partial_x^{\alpha+1}\rho =\partial_t\left( \partial_x^\alpha u_1\partial^{\alpha+1}_x\rho\right) +\partial_t(\partial_x^\alpha\rho)\partial_x^{\alpha+1}u_1-\partial_x\left[ \partial_t(\partial_x^\alpha\rho)\partial_x^\alpha u_1\right].
\emas
Substituting it into above and using the relation $\partial_x^{\alpha+1}E_1 = - \partial_x^\alpha\rho$, we have
\bmas
	&\quad \quad\frac{d}{dt} \int_{\mathbb{R}} \partial_x^\alpha u_1\partial_x^{\alpha+1}\rho\,dx +
	\int_{\mathbb{R}} \partial_x^\alpha\partial_t \rho\partial_x^{\alpha+1}u_1\,dx +  \int_{\mathbb{R}}\frac{P'(\rho+1)}{\rho +1}\left( \partial_x^{\alpha+1}\rho\right)^2 \,dx + \|\partial_x^{\alpha}\rho\|_{L^2}^2\\
	&= -\int_{\mathbb{R}}  \partial_x^\alpha u_1\partial_x^{\alpha+1}\rho\,dx - \int_{\mathbb{R}} f\partial_x^{\alpha+1}\rho\,dx.
\emas
Substituting $\partial_t\rho$ in the first equation of \eqref{em-2} into above, one has
\bmas
	&\quad\frac{d}{dt} \int_{\mathbb{R}} \partial_x^\alpha u_1\partial_x^{\alpha+1}\rho\,dx +  \int_{\mathbb{R}}\frac{P'(\rho+1)}{\rho +1}\left( \partial_x^{\alpha+1}\rho\right)^2 \,dx + \|\partial_x^{\alpha}\rho\|_{L^2}^2 \\
	&=\| \partial_x^{\alpha+1}u_1\|_{L^2}^2 - \int_{\mathbb{R}}  \partial_x^{\alpha+1}(\rho u_1)\partial_x^{\alpha+1}u_1\,dx -\int_{\mathbb{R}}  \partial_x^\alpha u_1\partial_x^{\alpha+1}\rho\,dx  - \int_{\mathbb{R}} f\partial_x^{\alpha+1}\rho\,dx\\
	&\le \| \partial_x^{\alpha+1}u_1\|_{L^2}^2 +\| \partial_x^{\alpha+1}(\rho u_1)\|_{L^2}\| \partial_x^{\alpha+1}u_1\|_{L^2}\\
	&\quad + \frac{1}{2}\| \partial_x^{\alpha+1} u_1 \|_{L^2}^2 +\frac12\| \partial_x^{\alpha}\rho \|_{L^2}^2 + \|  f\|_{L^2}\| \partial_x^{\alpha+1}\|_{L^2}.
\emas
For $\alpha =0$, we have
\bmas
	\| f \|_{L^2} \le \| u_1 \|_{L^\infty}\| \partial_x u_1\|_{L^2} + \| B_r \|_{L^\infty}\| \partial_x u_r\|_{L^2},
\emas
and for $\alpha \ge 1$, by Lemma \ref{sobolev-ineq}, we have
\bmas
	\| f\|_{L^2} &\le C\| u_1\|_{L^\infty}\| \partial_x^{\alpha+1} u_1\|_{L^2} +C\| \partial_x u_1 \|_{L^\infty}\| \partial_x^{\alpha} u_1\|_{L^2} + C\| B_r \|_{L^\infty}\| \partial_x^{\alpha} u_r\|_{L^2}  \\
	&\quad +C\| u_r \|_{L^\infty}\| \partial_x^{\alpha} B_r\|_{L^2}  +C\| \partial_x \rho \|_{L^\infty}\| \partial_x^{\alpha} \rho\|_{L^2}.
\emas
Taking summation of the above estimate over $ 0\le \alpha \le N-1$ , we obtain
\bma \label{dissipation-rho}
	&\quad \sum_{ \alpha=0}^{N-1} 	\frac{d}{dt}\int_{\mathbb{R}} \partial_x^\alpha u_1\partial_x^{\alpha+1}\rho\,dx + \| \rho \|_{H^N}^2  \nnm\\
	&\le C_1 \| u_1\|_{H^N}^2 + C_1\| (\rho, u_1, u_r, B_r)\|_{H^2} \left(\| (\rho,u_1,u_r)\|_{H^N}^2 + \| \partial_x B_r\|_{H^{N-2}}^2\right).
\ema

Next, for $0\le \alpha \le N-1$, applying $\partial_x^\alpha$ to \eqref{em-2} and taking the inner product of the second and fourth equations with $\partial_x^\alpha E_1$ and $\partial_x^\alpha u_1$ respectively. Then, we add them up and use equation $\partial_x^{\alpha+1}E_1=-\partial_x^\alpha \rho$, one has
\bmas
	&\quad \frac{d}{dt}\int_{\mathbb{R}} \partial_x^\alpha u_1\partial_x^\alpha E_1\,dx + \| \partial_x^\alpha E_1 \|_{L^2}^2 +\int_{\mathbb{R}}\frac{P'(\rho+1)}{\rho+1} \left(\partial_x^{\alpha+1}\rho\right)^2\,dx \\
	&= -\int_{\mathbb{R}}  \partial_x^\alpha u_1 \partial_x^\alpha E_1\,dx - \int_{\mathbb{R}} f\partial_x^\alpha E_1 \,dx +\int_{\mathbb{R}} \partial_x^\alpha [(\rho +1)u_1]\partial_x^\alpha u_1 \,dx\\
	&\le \frac{1}{2}\| \partial_x^\alpha u_1\|_{L^2}^2 + \frac{1}{2}\| \partial_x^\alpha E_1\|_{L^2}^2 + \| f\|_{L^2}\| \partial_x^\alpha E_1\|_{L^2} + \| \partial_x^\alpha u_1\|_{L^2}^2 +\| \partial_x^\alpha (\rho u_1)\|_{L^2}\| \partial_x^\alpha u_1\|_{L^2}.
\emas
After taking summation of the above estimate over $0\le \alpha < N-1$, it yields
\bma \label{dissipation-e1}
	&\quad \sum_{ \alpha=0}^{N-1} 	\frac{d}{dt}\int_{\mathbb{R}} \partial_x^\alpha u_1\partial_x^{\alpha}E_1\,dx + \| E_1 \|_{H^{N-1}}^2 +\| \rho \|_{H^{N-1}}^2 \nnm \\
	&\le C_2 \| u_1\|_{H^{N-1}}^2 + C_2\| (\rho, u_1, u_r, B_r)\|_{H^2} \left(\| (\rho,u_1,u_r)\|_{H^N}^2 + \| E_1 \|_{H^{N-1}}^2 + \| \partial_x B_r\|_{H^{N-2}}^2\right).
\ema

And next, for $0\le \alpha \le N-1$,  applying $\partial_x^\alpha$ to the third equation of \eqref{em-2}, taking the inner product  with $\partial_x^\alpha E_r$ in $L^2{(\mathbb{R})}$ and then using the fifth equation of \eqref{em-2} yields
\bmas
	&\quad \frac{d}{dt}\int_{\mathbb{R}} \partial_x^\alpha u_r\cdot\partial_x^\alpha E_r\,dx + \| \partial_x^\alpha E_r \|_{L^2}^2 \\
	&= -\int_{\mathbb{R}} (\mathbb{I}_2-\mathbb{O}_1)\partial_x^\alpha u_r \cdot \partial_x^\alpha E_r\,dx + \int_{\mathbb{R}} \mathbb{O}_1\partial_x^{\alpha+1}B_r\cdot \partial_x^\alpha u_r\,dx  \\
	&\quad- \int_{\mathbb{R}} \partial_x^\alpha(u_1\partial_x u_r+  u_1\mathbb{O}_1B_r)\cdot\partial_x^\alpha E_r\,dx +  \int_{\mathbb{R}} \partial_x^\alpha(\rho u_r)\cdot\partial_x^\alpha u_r \,dx + \| \partial_x^\alpha u_r\|_{L^2}^2.
\emas
We have the following estimates:
\bmas
	&-\int_{\mathbb{R}} (\mathbb{I}_2-\mathbb{O}_1 )\partial_x^\alpha u_r \cdot \partial_x^\alpha E_r\,dx \le 2\| \partial_x^\alpha u_r\|_{L^2}^2 + \frac{1}{2}\| \partial_x^\alpha E_r\|_{L^2}^2,\\
	&\int_{\mathbb{R}} \mathbb{O}_1\partial_x^{\alpha+1}B_r\cdot\partial_x^\alpha u_r\,dx \le
	\begin{cases}
		\frac{C}{\epsilon}\| u_r\|_{L^2}^2 + \epsilon\| \partial_x B_r\|_{L^2}^2,& \alpha=0,\\
		\frac{C}{\epsilon}\| \partial_x^{\alpha+1} u_r\|_{L^2}^2 + \epsilon\| \partial_x^\alpha B_r\|_{L^2}^2,& 1\le \alpha \le N-1 .
	\end{cases}
\emas
Here, $\epsilon >0$ is a constant. For $\alpha=0$, we have
$$
	\| u_1\partial_x u_r+u_1\mathbb{O}_1B_r\|_{L^2} \le \|u_1\|_{L^\infty}\| \partial_x u_r\|_{L^2} + \|B_r\|_{L^\infty}\|u_1\|_{L^2},
$$
and for $\alpha \ge 1$, one has
\bmas
	\| \partial_x^\alpha (u_1\partial_x u_r+u_1\mathbb{O}_1B_r)\|_{L^2} &\le \| u_1\|_{L^\infty}\| \partial_x^{\alpha+1}u_r\|_{L^2} +\| \partial_x u_r\|_{L^\infty}\| \partial_x^\alpha u_1\|_{L^2}\\
	&\quad + \| B_r\|_{L^\infty}\| \partial_x^{\alpha}u_1\|_{L^2} +\| u_1\|_{L^\infty}\| \partial_x^\alpha B_r\|_{L^2}.
\emas
Further taking summation over $0\le \alpha \le N-1$, we have
\bma \label{dissipation-er}
	&\quad \sum_{ \alpha=0}^{N-1}\frac{d}{dt}\int_{\mathbb{R}} \partial_x^\alpha u_r \cdot \partial_x^\alpha E_r\,dx +\| E_r \|_{H^{N-1}}^2 \nnm\\
	&\le \frac{C_3}{\epsilon}\| u_r\|_{H^N}^2 +\epsilon\| \partial_x B_r\|_{H^{N-2}}^2 \nnm\\
	&\quad +C_3\|(\rho,u_1,u_r,B_r)\|_{H^2}\left( \| (u_1,u_r)\|_{H^N} + \| (\rho,E_r)\|_{H^{N-1}}^2 + \| \partial_x B_r\|_{H^{N-2}}^2\right).
\ema
Thus, \eqref{dissipation-rho-E} follows from taking summation of \eqref{dissipation-rho}, \eqref{dissipation-e1} and \eqref{dissipation-er}. Therefore, we complete this proof.
\end{proof}
			
\begin{lem}
It holds that
\bma \label{dissipation-Br}
\begin{split}
	&\quad -\sum_{ \alpha=0}^{N-2} \frac{d}{dt} \int_{\mathbb{R}} \partial_x^\alpha E_r\cdot\mathbb{O}_1\partial_x^{\alpha+1}B_r\,dx  + \| \partial_xB_r\|_{H^{N-2}}^2 \\
	&\le C  (\| u_r\|_{H^{N-2}}^2 + \| \partial_xE_r\|_{H^{N-2}}^2)  + C \| (\rho,u_r)\|_{L^\infty} (\| (\rho,u_r)\|_{H^{N-2}}^2 + \| \partial_xB_r\|_{H^{N-2}}^2 ) .
\end{split}
\ema
\end{lem}
			
\begin{proof}
After applying $\partial_x^\alpha$ to \eqref{em-2}, we take the inner product of the fifth and sixth equations with $-\mathbb{O}_1\partial_x^{\alpha+1} B_r$ and $-\mathbb{O}_1\partial_x^{\alpha+1}E_r$  in $L^2{(\mathbb{R})}$. Then, adding up these two equations, it yields
\bmas
	&\quad-\frac{d}{dt}\int_{\mathbb{R}} \partial_x^\alpha E_r\mathbb{O}_1\partial_x^{\alpha+1}B_r\,dx + \| \partial_x^{\alpha+1}B_r\|_{L^2}^2 + \int_{\mathbb{R}}  \partial_x^\alpha u_r\mathbb{O}_1\partial_x^{\alpha+1}B_r\,dx \\
	&=\| \partial_x^{\alpha+1}E_r\|_{L^2}^2 -  \int_{\mathbb{R}} \partial_x^\alpha(\rho u_r) \mathbb{O}_1\partial_x^{\alpha+1}B_r\,dx.
\emas
We have the estimates as follow:
\bmas
	\int_{\mathbb{R}}  \partial_x^\alpha u_r\mathbb{O}_1\partial_x^{\alpha+1}B_r\,dx &\le \frac12 \| \partial_x^{\alpha+1}B_r\|_{L^2}^2 + C\| \partial_x^\alpha u_r\|_{L^2}^2,\\
	\int_{\mathbb{R}} \partial_x^\alpha(\rho u_r)\mathbb{O}_1\partial_x^{\alpha+1}B_r\,dx &\le C\left( \| \rho \|_{L^\infty}\| \partial_x^\alpha u_r\|_{L^2} + \| u_r \|_{L^\infty}\| \partial_x^\alpha \rho \|_{L^2}\right)\| \partial_x^{\alpha+1}B_r\|_{L^2}.
\emas
Thus, \eqref{dissipation-Br} follows from taking summation over $0\le \alpha \le N-2$.
\end{proof}
			
\begin{prop} \label{a priori estimate}
Let $N\ge2$ and suppose $W(t,x)$ be a solution to \eqref{em-2} with initial data $W(0)$, satisfying  for any given $T>0$ with
\be
	\sup_{0\le t \le T}\| W(t)\|_{H^N} \le \delta,
\ee
where $\delta$ is a positive constant. Then, the following a priori estimate holds for $t\in[0,T]$:
\be
	\| W(t) \|_{H^N}^2 + \int_0^t \left( \| (\rho,u_1,u_r)\|_{H^N}^2 + \| (E_1,E_r)\|_{H^{N-1}}^2 + \| \partial_x B_r\|_{H^{N-2}}^2\right) \,d\tau \le C\| W(0)\|_{H^N}^2.
\ee
\end{prop}
			
\begin{proof}
Firstly, taking  $\eqref{zero-energy}+ \sum_{ \alpha=1}^{N}\eqref{alpha-energy}$, it yields
\be\label{estimate1}
	\| W(t)\|_{H^N}^2 + \int_0^t \| (u_1,u_r)(s)\|_{H^N}^2 ds \le C\| W(0)\|_{H^N}^2+CH_1\int_0^t\| (\rho,u_1,u_r)(s)\|_{H^N}^2ds,
\ee
where
$$
	H_1 := \sup_{0 \le s \le t}\left( M_2 + \| B_r\|_{L^\infty}+\| \partial_x(E_1,E_r,B_r)\|_{H^{N-1}}\right),
$$
with $M_2 = \| \partial_x u_1\|_{L^\infty}\| \rho\|_{L^\infty} +\| \partial_x\rho\|_{L^\infty}\| u_1\|_{L^\infty}+\| \partial_x (\rho,u_1,u_r)\|_{L^\infty}.$
				
Next, integrating \eqref{dissipation-rho-E} and \eqref{dissipation-Br} with respect to $t$, it yields
\bma \label{estimate2}
	&\quad \sum_{ \alpha=0}^{N-1} \left(\int_{\mathbb{R}} \partial_x^\alpha u_1\partial_x^{\alpha+1}\rho+\partial_x^\alpha u_1\partial_x^{\alpha}E_1 + \partial_x^\alpha u_r\cdot\partial_x^{\alpha}E_r\,dx\right) + \int_0^t \| \rho\|_{H^N}^2 + \| (E_1,E_r)\|_{H^{N-1}}^2ds  \nnm\\
	&\le C\| (\rho,u_1,u_r,E_1,E_r)(0)\|_{H^N}^2 + \frac{C}{\epsilon}\int_0^t\| (u_1,u_r)\|_{H^N}^2ds + \epsilon\int_0^t \| \partial_x B_r\|_{H^{N-2}}^2ds \nnm\\
	&\quad + CH_2\int_0^t \| (\rho,u_1,u_r)\|_{H^N}^2 + \| (E_1,E_r)\|_{H^{N-1}}^2 + \| \partial_xB_r\|_{H^{N-2}}^2ds,
\ema
and
\bma \label{estimate3}
	&\quad -\sum_{ \alpha=0}^{N-2}\int_{\mathbb{R}} \partial_x^\alpha E_r\cdot\mathbb{O}_1\partial_x^{\alpha+1}B_r\,dx + \int_0^t\| \partial_xB_r\|_{H^{N-2}}^2ds \nnm\\
	&\le C\| (E_r,\partial_xB_r)(0)\|_{H^{N-2}}^2 + C\int_0^t\| u_r\|_{H^{N-2}}^2 ds  +C\int_0^t\| \partial_xE_r\|_{H^{N-2}}^2 ds\nnm\\
	&\quad  +CH_2 \int_0^t\| (\rho,u_1)\|_{H^{N-2}}\| \partial_xB_r\|_{H^{N-2}}ds,
\ema
where
$$
	H_2:=\sup_{0\le s\le t}\| (\rho,u_1,u_r,B_r)\|_{H^2}.
$$
And then, summing up the estimates $\eqref{estimate1}, \kappa_1\times\eqref{estimate2}$ and $\kappa_2\times\eqref{estimate3}$, we have
\bmas
	&\quad \| W\|_{H^N}^2 + \kappa_1\sum_{ \alpha=0}^{N-1} \left(\int_{\mathbb{R}} \partial_x^\alpha u_1\partial_x^{\alpha+1}\rho+\partial_x^\alpha u_1\partial_x^{\alpha}E_1 + \partial_x^\alpha u_r\cdot\partial_x^{\alpha}E_r\,dx\right)\\
 &\quad-\kappa_2\sum_{ \alpha=0}^{N-2}\int_{\mathbb{R}} \partial_x^\alpha E_r\cdot\mathbb{O}_1\partial_x^{\alpha+1}B_r\,dx\\
	&\quad +\int_0^t  \| (u_1,u_r) \|_{H^N}^2 + \kappa_1\| \rho\|_{H^N}^2 +  \kappa_1\| (E_1,E_r)\|_{H^{N-1}}^2 + \kappa_2\| \partial_xB_r\|_{H^{N-2}}^2ds\\
	&\le C\| W(0)\|_{H^N}^2 +\frac{\kappa_1+\epsilon\kappa_2}{\epsilon}C\int_0^t\| (u_1,u_r) \|_{H^N}^2ds +C\kappa_2\int_0^t\| \partial_xE_r\|_{H^{N-2}}^2ds \\
	&\quad + \epsilon\kappa_1\int_0^t \| \partial_x B_r\|_{H^{N-2}}^2ds + C(H_1+H_2 )\int_0^t \| (\rho,u_1,u_r)\|_{H^N}^2 + \| (E_1,E_r)\|_{H^{N-1}}^2 + \| \partial_xB_r\|_{H^{N-2}}^2ds.
\emas
Here, $0< \kappa_2 < \kappa_1$ are constants. We take $\kappa_1 = \epsilon^2$ such that $\kappa_1+\kappa_2 \ll 1$ and $\epsilon\kappa_1 < \kappa_2$ for $0<\epsilon \ll 1$. Thus, we obtain
\bmas
	&\quad\| W\|_{H^N}^2 + \int_0^t  \| (u_1,u_r)\|_{H^N}^2 + \| \rho\|_{H^N}^2 +  \| (E_1,E_r)\|_{H^{N-1}}^2 + \| \partial_xB_r\|_{H^{N-2}}^2ds\\
	&\le  C_1 \| W(0)\|_{H^N}^2+ C_1(H_1+H_2)\int_0^t \| (\rho,u_1,u_r)\|_{H^N}^2 + \| (E_1,E_r)\|_{H^{N-1}}^2 + \| \partial_xB_r\|_{H^{N-2}}^2ds.
\emas
We notice that $ H_1,H_2\le C_2 \displaystyle\sup_{0\le s\le t}\| W\|_{H^N} $ for $N\ge 2$. Therefore, we have
\bmas
	&\quad\| W\|_{H^N}^2 + \int_0^t  \| \rho\|_{H^N}^2 + \| (u_1,u_r)\|_{H^N}^2 +  \| (E_1,E_r)\|_{H^{N-1}}^2 + \| \partial_xB_r\|_{H^{N-2}}^2ds\\
	&\le  C_1 \| W(0)\|_{H^N}^2 +C_3\sup_{0\le s\le t}\| W\|_{H^N} \int_0^t \| (\rho,u_1,u_r)\|_{H^N}^2 +  \| (E_1,E_r)\|_{H^{N-1}}^2 + \| \partial_xB_r\|_{H^{N-2}}^2 \,ds.
\emas
We take $\delta\le\frac{1}{2C_3}$. Under the assumption $\displaystyle\sup_{0\le s\le t}\| W\|_{H^N}^2 \le \delta$, we get the following estimate:
\be
	\| W \|_{H^N}^2 + \int_0^t \left( \| (\rho,u_1,u_r)\|_{H^N}^2 + \| (E_1,E_r)\|_{H^{N-1}}^2 + \| \partial_x B_r\|_{H^{N-2}}^2\right) \,d\tau \le C_4\| W(0)\|_{H^N}^2.
\ee
And then, we done with the proof of Proposition \eqref{a priori estimate}.
\end{proof}

Before giving the proof of Theorem \ref{global-tm}, we show the following local existence result of system \eqref{em-2} which can refer \cite{Kato} and \cite{Taylor} (see section 1 and section 2, chapter 16).
\begin{lem} \label{local}
	Suppose that $W(0) \in H^N(\mathbb{R})$ satisfies \eqref{init-compatible} for $N>\frac{3}{2}$. Then, there is $T>0$ such that the system \eqref{em-2}--\eqref{init} admits a unique solution on $[0,T)$ with
	$$
		W \in C([0,T);H^N(\mathbb{R}))\cap C^1([0,T);H^{N-1}(\mathbb{R})).
	$$
\end{lem}
 \begin{proof}[\textbf{Proof of Theorem \ref{global-tm}} ]
	The global existence of our solution follows in the standard way by combining the continuation argument base on a local in time existence given in Lemma \ref{local} and the a priori estimate given in Proposition \ref{a priori estimate}. And the global solution satisfies energy estimate \eqref{energy-estimate} for all $t\ge 0$.
\end{proof}

\subsection{Time Decay Rates of Nonlinear System}
\begin{thm} \label{nonlinear-decay}
Assume that $U(0)=(\rho_0,u_{10},u_{r0},E_{10},E_{r0},B_{r0})^{T}\in H^{N}(\mathbb{R})\cap L^1(\mathbb{R})$ with $N\ge 4$, and $\| U(0)\|_{H^N}\le \delta_0$ with $\delta_0 > 0$ being sufficiently small. Let $U=(\rho,u_1,u_r,E_1,E_r,B_r)^{T}$ be a solution of the Euler-Maxwell system \eqref{re-em_1}--\eqref{re-initial}. Then, it holds that for $\alpha=0,1$,
\bma
	\| \partial_x^\alpha(\rho,u_1)(t)\|_{L^2} &\le C\delta_0(1+t)^{-\frac54-\frac{\alpha}{2}}, \label{time_decay_1}\\
	\| \partial_x^\alpha E_1(t)\|_{L^2} &\le C\delta_0(1+t)^{-\frac54},\label{time_decay_2}\\
	\| \partial_x^\alpha (u_r,E_r)(t)\|_{L^2} &\le C\delta_0(1+t)^{-\frac34-\frac{\alpha}{2}}, \label{time_decay_3}\\
	\|\partial_x^\alpha B_r(t) \|_{L^2} &\le C\delta_0(1+t)^{-\frac14-\frac{\alpha}{2}}. \label{time_decay_4}
\ema
\end{thm}			
			
\begin{proof}
Let $U_1=(u_1,E_1)^T$ and $U_r=(u_r,E_r,B_r)^T$. Then, by the Duhamel's principle, one has
\bma
	U_1(t,x) &= G_f(t) \ast U_1(0) + \int_0^t G_f(t-s) \ast H_1(s) ds, \label{duhamel_1}\\
	U_r(t,x) &= G_e(t) \ast U_r(0) + \int_0^t G_e(t-s) \ast H_r(s) ds. \label{duhamel_r}
\ema
Here, $H_1=(h_1,h_2)^{T},~H_r=(h_3,h_4,0)^{T}$ with
\bmas
	h_1 &=- u_1\partial_x u_1 -  \left( \frac{P'(\rho+1)}{\rho+1} - P'(1) \right)\partial_x \rho +  u_r\cdot\mathbb{O}_1B_r,\\
	h_2 &= \rho u_1,\\
	h_3 &=- u_1\partial_x u_r -  u_1\mathbb{O}_1B_r,\\
	h_4 &=  \rho u_r.
\emas
Now, we defined a functional $Q(t)$ for any $t>0$ as
\bma \label{ansatz}
\begin{split}
	Q(t) = &\sup_{0 \le s \le t}
	\bigg\{ \sum_{ \alpha=0,1}\((1+s)^{\frac54+\frac{\alpha}2}\| \partial_x^\alpha(\rho,u_1) (s)\|_{L^2} + (1+s)^{\frac54}\| \partial_x^\alpha E_1(s)\|_{L^2}\)  \\
	& \qquad \quad + \sum_{ \alpha=0,1}\((1+s)^{\frac34+\frac{\alpha}{2}}\| \partial_x^\alpha (u_r,E_r )(s)\|_{L^2} + (1+s)^{\frac14+\frac{\alpha}{2}}\| \partial_x^\alpha B_r(s)\|_{L^2}\) \\
	& \qquad \quad + \| \partial_x^2(\rho,u_1,u_r,B_r)(s)\|_{H^2}\bigg\} .
\end{split}
\ema
In what follows, we should prove that
\be \label{claim} Q(t) \le C\delta_0.\ee
Once it holds, we will obtain the time decay rates \eqref{time_decay_1}--\eqref{time_decay_4}. Moreover, by applying the  Gagliardo-Nirenberg interpolation inequality,  we have
\bmas
\| \partial_x^2 \rho\|_{L^2} &\le C\| \partial_x^4\rho\|_{L^2}^{\frac13} \| \partial_x\rho\|_{L^2}^{\frac23}\le C(1+t)^{-\frac76}Q(t),\\
\| \partial_x^3 \rho\|_{L^2} &\le C\| \partial_x^4\rho\|_{L^2}^{\frac23} \| \partial_x\rho\|_{L^2}^{\frac13}\le C(1+t)^{-\frac{7}{12}}Q(t).
\emas
Similarly, we also have
\bmas
&\| \partial_x^2 u_1\|_{L^2} \le C(1+t)^{-\frac76}Q(t),\quad \| \partial_x^3 u_1\|_{L^2} \le C(1+t)^{-\frac{7}{12}}Q(t),\\
& \| \partial_x^2 u_r\|_{L^2} \le C(1+t)^{-\frac56}Q(t),\quad \| \partial_x^3 u_r\|_{L^2}  \le (1+t)^{-\frac{5}{12}}Q(t) ,\\
&\| \partial_x^2 B_r\|_{L^2} \le C(1+t)^{-\frac12}Q(t),\quad \| \partial_x^3 B_r\|_{L^2}  \le C(1+t)^{-\frac14}Q(t) .
\emas
The Gagliardo-Nirenberg interpolation inequality is stated as: Let $1\le q,r\le \infty$, for $0\le j < m$, it holds
$$
\| \partial^j u\|_{L^p} \le C\| \partial^m u\|_{L^r}^a\| u\|_{L^q}^{1-a},	\quad u\in L^q(\mathbb{R}^n),
$$
where
$$
\frac{1}{p} = \frac{j}{n} +a\(\frac{1}{r}-\frac{m}{n}\)+(1-a)\frac{1}{q}, \qquad \frac{j}{m} \le a \le 1.
$$
Firstly, we give some estimates for the nonlinear terms $h_1$ and $h_2$ for $0\le s \le t$ in terms of $Q(t)$ as
\bmas
	\| h_1(t) \|_{L^2} &\le C( \| u_1\|_{L^\infty}\|\partial_x  u_1\|_{L^2} + \|  \rho\|_{L^\infty} \|\partial_x \rho\|_{L^2}+\| B_r\|_{L^\infty}\| u_r\|_{L^2})\\
		&\le C(1+t)^{-\frac54}Q^2(t),\\
	\| \partial_x h_1 \|_{L^2} &\le  C( \| u_1\|_{L^\infty}\|\partial_x^2  u_1\|_{L^2} +\| \partial_x u_1\|_{L^\infty}\| \partial_x u_1\|_{L^2} + \|  \rho\|_{L^\infty} \|\partial_x^2 \rho\|_{L^2} \\
	&\quad +\| \partial_x\rho\|_{L^\infty}\| \partial_x \rho\|_{L^2}+\| u_r\|_{L^\infty}\| \partial_x B_r\|_{L^2} +\| B_r\|_{L^\infty}\| \partial_x u_r\|_{L^2})\\
& \le C(1+t)^{-\frac74}Q^2(t),\\
	\| h_2(t)\|_{L^2} &\le C\| \rho \|_{L^\infty}\| u_1\|_{L^2} \le C(1+t)^{-\frac{11}4}Q^2(t),\\
	\| \partial_x h_2(t) \|_{L^2} &\le ( \|  u_1\|_{L^\infty}\| \partial_x\rho\|_{L^2} +\|  \rho\|_{L^\infty}\| \partial_x u_1 \|_{L^2})
	\le C(1+t)^{-\frac{13}4}Q^2(t),\\
	\| \partial_x^2 h_2(t) \|_{L^2} &\le ( \|  u_1\|_{L^\infty}\| \partial_x^2\rho\|_{L^2} +\| \rho \|_{L^\infty}\| \partial_x^2 u_1 \|_{L^2}) \le C(1+t)^{-\frac{8}{3}}Q^2(t),
\emas
where we use the following inequality
\bmas
	&\| u_1\|_{L^\infty} \le \| u_1\|_{L^2}^{\frac{1}{2}}\| \partial_x u_1\|_{L^2}^{\frac{1}{2}} \le (1+t)^{-\frac32}Q(t), \\
	&\| B_r\|_{L^\infty} \le \| B_r\|_{L^2}^{\frac{1}{2}}\| \partial_x B_r\|_{L^2}^{\frac{1}{2}} \le (1+t)^{-\frac{1}{2}}Q(t), \\
	&\| u_r\|_{L^\infty} \le \| u_r\|_{L^2}^{\frac{1}{2}}\| \partial_x u_r\|_{L^2}^{\frac{1}{2}} \le (1+t)^{-1}Q(t).
\emas

According to \eqref{duhamel_1}, we obtain by Lemma \ref{Gf-convolution} that
\bmas
	\| u_1\|_{L^2} &\le\| G_f^{11}\ast u_{10}\|_{L^2} + \| G_f^{12}\ast E_{10}\|_{L^2} \\
	&\quad +\int_0^t \| G_f^{11}(t-s)\ast h_1(s)\|_{L^2} + \| G_f^{12}(t-s)\ast h_2(s)\|_{L^2} ds\\
	&\le Ce^{-\frac{t}{2}}(\| u_{10}\|_{L^2} +\| E_{10}\|_{L^2} +\| \partial_xE_{10}\|_{L^2}) \\
	&\quad +C\int_0^te^{-\frac{t-s}{2}}(\| h_1(s) \|_{L^2} + \| h_2(s) \|_{L^2} + \| \partial_x h_2(s) \|_{L^2}) ds \\
	&\le Ce^{-\frac{t}{2}}\delta_0 + C\int_0^t e^{-\frac{t-s}{2}}(1+s)^{-\frac54}Q^2(t)ds \\
	&\le Ce^{-\frac{t}{2}}\delta_0 + C(1+t)^{-\frac54}Q^2(t),\\
	\| \partial_x u_1\|_{L^2} &\le\| \partial_x (G_f^{11}\ast u_{10})\|_{L^2} + \| \partial_x (G_f^{12}\ast E_{10})\|_{L^2} \\
	&\quad +\int_0^t \| \partial_x (G_f^{11}(t-s)\ast h_1(s))\|_{L^2} + \| \partial_x (G_f^{12}(t-s)\ast h_2(s))\|_{L^2} ds\\
	&\le Ce^{-\frac{t}{2}}(\| (\partial_x u_{10},\partial_x E_{10})\|_{L^2} +\| \partial_x^2E_{10}\|_{L^2}) \\
	&\quad +C\int_0^te^{-\frac{t-s}{2}}( \| \partial_x h_1(s) \|_{L^2} + \| \partial_x h_2(s) \|_{L^2} + \| \partial_x^2 h_2(s) \|_{L^2}) ds \\
	&\le Ce^{-\frac{t}{2}}\delta_0 + C\int_0^t e^{-\frac{t-s}{2}}(1+s)^{-\frac74}Q^2(t)ds \\
	&\le Ce^{-\frac{t}{2}}\delta_0 + C(1+t)^{-\frac74}Q^2(t).
\emas
For $E_1$ and $\rho=-\dx E_1,$ we have
\bmas
	\| E_1\|_{L^2} &\le \| G_f^{21}\ast u_{10}\|_{L^2} + \| G_f^{22}\ast E_{10}\|_{L^2}\\
	&\quad+\int_0^t \| G_f^{21}(t-s)\ast h_1(s)\|_{L^2} + \| G_f^{22}(t-s)\ast h_2(s)\|_{L^2}ds\\
	&\le Ce^{-\frac{t}{2}}(\| u_{10}\|_{L^2} + \| E_{10}\|_{L^2}) + C\int_0^t e^{-\frac{t-s}{2}}(\| h_1(s)\|_{L^2} +\| h_2(s)\|_{L^2})ds\\
	&\le Ce^{-\frac{t}{2}}\delta_0 + C(1+t)^{-\frac54}Q^2(t),\\
	\| \partial_x E_1 \|_{L^2}& \le \| \partial_x (G_f^{21}\ast u_{10})\|_{L^2} + \| \partial_x(G_f^{22}\ast E_{10})\|_{L^2}\\
	&\quad +\int_0^t \| \partial_x (G_f^{21}(t-s)\ast h_1(s))\|_{L^2} + \| \partial_x(G_f^{22}(t-s)\ast h_2(s))\|_{L^2}ds\\
	&\le Ce^{-\frac{t}{2}}(\| u_{10} \|_{L^2} +\| \partial_x E_{10}\|_{L^2}) \\
	&\quad+ C\int_0^t e^{-\frac{t-s}{2}}(\| h_1(s)\|_{L^2} + \| \partial_x h_2(s)\|_{L^2}) ds\\
	&\le Ce^{-\frac{t}{2}}\delta_0 + C(1+t)^{-\frac54}Q^2(t),\\
	\| \partial_x \rho \|_{L^2} &\le \| \partial_x^2 (G_f^{21}\ast u_{10})\|_{L^2} + \| \partial_x^2(G_f^{22}\ast E_{10})\|_{L^2}\\
	&\quad +\int_0^t \| \partial_x^2 (G_f^{21}(t-s)\ast h_1(s))\|_{L^2} + \| \partial_x^2(G_f^{22}(t-s)\ast h_2(s))\|_{L^2}ds\\
	&\le Ce^{-\frac{t}{2}}(\| \partial_x u_{10} \|_{L^2}+\| \partial_x^2 E_{10}\|_{L^2}) \\
	&\quad+ C\int_0^t e^{-\frac{t-s}{2}}( \| \partial_x h_1(s) \|_{L^2} +\| \partial_x^2 h_2(s)\|_{L^2}) ds\\
	&\le Ce^{-\frac{t}{2}}\delta_0 + C(1+t)^{-\frac74}Q^2(t).
\emas
Next, we provide some estimates of nonlinear term $h_3$ and $h_4$ as following
\bmas
	\| h_3(t)\|_{L^1}
	&\le C( \| u_1\|_{L^2}\| \partial_x u_r\|_{L^2}+\| u_1\|_{L^2}\| B_r\|_{L^2}) \le C(1+t)^{-\frac32}Q^2(t),\\
	\| h_3(t)\|_{L^2}
	&\le C( \| u_1\|_{L^\infty}\|  \partial_x u_r\|_{L^2}+\| B_r\|_{L^\infty}\| u_1\|_{L^2}) \le C(1+t)^{-\frac74}Q^2(t),\\
	\| \partial_x^{\alpha} h_3(s) \|_{L^2} &\le C(  \| u_1\|_{L^\infty}\| \partial_x^{\alpha+1} u_r\|_{L^2}   + \| \partial_x u_r\|_{L^\infty}\| \partial_x^{\alpha} u_1\|_{L^2})\\
	&\quad +C( \| u_1\|_{L^\infty}\| \partial_x^{\alpha} B_r\|_{L^2} + \| B_r\|_{L^\infty}\| \partial_x^{\alpha} u_1\|_{L^2})\\
&\le C (1+t)^{-\frac{17}6+ \frac{ 7\alpha}{12}}Q^2(t),\quad \alpha=1,2,3,
\\
	\| h_4(t)\|_{L^1} & \le C\| \rho \|_{L^2}\| u_r\|_{L^2} \le C(1+t)^{-2}Q^2(t),\\
	\| h_4(t)\|_{L^2} &\le C \|\rho \|_{L^2} \| u_r \|_{L^\infty}\le C(1+t)^{-\frac94}Q^2(t),\\
	\| \partial_x^\alpha h_4(t)\|_{L^2} &\le C( \| \rho \|_{L^\infty}\| \partial_x^\alpha u_r\|_{L^2} +\| u_r \|_{L^\infty}\| \partial_x^\alpha \rho \|_{L^2} )\\
	&\le C(1+t)^{-\frac{10}3+ \frac{ 7\alpha}{12}}Q^2(t),\quad \alpha=1,2,3,4.
\emas
According to \eqref{duhamel_r} and the corollary \ref{Ur_linear_time}, we have
\bmas
	\| u_r\|_{L^2}  &\le \| G_e^{11}\ast u_{r0}\|_{L^2} +\| G_e^{12}\ast E_{r0}\|_{L^2} +\| G_e^{13}\ast B_{r0}\|_{L^2}\\
	&\quad +\int_0^t \| G_e^{11}(t-s)\ast h_3(s)\|_{L^2} + \| G_e^{12}(t-s)\ast h_4(s)\|_{L^2}ds\\
	&\le C(1+t)^{-\frac34}( \| (u_{r0},E_{r0},B_{r0})\|_{L^1\cap L^2}+\| \partial_x(u_{r0},E_{r0},B_{r0})\|_{L^2})\\
	&\quad + C\int_0^t(1+t-s)^{-\frac54}\| h_3(s)\|_{L^1\cap L^2} +  (1+t-s)^{-1}\| h_3(s)\|_{L^2} ds\\
	&\quad +C\int_0^t(1+t-s)^{-\frac54}\| h_4(s)\|_{L^1\cap L^2} + (1+t-s)^{-1}\| \partial_x h_4(s)\|_{L^2}ds\\
	&\le C(1+t)^{-\frac34}\delta_0\\
	&\quad + CQ^2(t)\int_0^t(1+t-s)^{-\frac54}(1+s)^{-\frac32} + (1+t-s)^{-1}(1+s)^{-\frac74} ds\\
	&\quad + CQ^2(t)\int_0^t(1+t-s)^{-\frac54}(1+s)^{-2} + (1+t-s)^{-1}(1+s)^{-\frac{11}4} ds\\
	&\le C(1+t)^{-\frac34}\delta_0 + C(1+t)^{-1}Q^2(t),\\
	\| \partial_x u_r\|_{L^2}  &\le \| \partial_x(G_e^{11}\ast u_{r0})\|_{L^2} +\| \partial_x (G_e^{12}\ast E_{r0})\|_{L^2} +\| \partial_x (G_e^{13}\ast B_{r0})\|_{L^2}\\
	&\quad +\int_0^t \| \partial_x (G_e^{11}(t-s)\ast h_3(s))\|_{L^2} + \| \partial_x (G_e^{12}(t-s)\ast h_4(s))\|_{L^2}ds\\
	&\le C(1+t)^{-\frac54}( \| (u_{r0},E_{r0},B_{r0})\|_{L^1\cap L^2}+\| \partial_x^3(u_{r0},E_{r0},B_{r0})\|_{L^2})\\
	&\quad  + C\int_0^t(1+t-s)^{-\frac74}\| h_3(s)\|_{L^1\cap L^2}+(1+t-s)^{-\frac32}\| \partial_x^2 h_3(s)\|_{L^2}ds \\
	&\quad  +C\int_0^t(1+t-s)^{-\frac74}\|h_4(s)\|_{L^1\cap L^2}+(1+t-s)^{-\frac32}\| \partial_x^3 h_4(s)\|_{L^2}ds\\
	&\le C(1+t)^{-\frac54}\delta_0\\
	&\quad + CQ^2(t)\int_0^t(1+t-s)^{-\frac74}(1+s)^{-\frac32}ds +(1+t-s)^{-\frac32}(1+s)^{-\frac{5}3} ds\\
	&\quad + CQ^2(t)\int_0^t (1+t-s)^{-\frac74}(1+s)^{-2}+(1+t-s)^{-\frac32}(1+s)^{-\frac{19}{12}} ds\\
	&\le C(1+t)^{-\frac54}\delta_0 + C(1+t)^{-\frac32} Q^2(t).
\emas
For $E_r(t)$, we have
\bmas
	\| E_r\|_{L^2} &\le \|  G_e^{21}\ast  u_{r0}\|_{L^2} + \|  G_e^{22}\ast  E_{r0}\|_{L^2} + \|  G_e^{23}\ast  B_{r0}\|_{L^2} \\
	&\quad +\int_0^t \|  G_e^{21}(t-s)\ast  h_3(s)\|_{L^2} + \| G_e^{22}(t-s)\ast h_4(s)\|_{L^2}ds\\
	&\le C(1+t)^{-\frac34}( \| (u_{r0},E_{r0},B_{r0})\|_{L^1\cap L^2} +\| \partial_x^2(u_{r0},E_{r0},B_{r0})\|_{L^2}) \\
	&\quad +C\int_0^t(1+t-s)^{-\frac54}\| h_3(s)\|_{L^1\cap L^2} + (1+t-s)^{-1}\| \partial_x h_3(s)\|_{L^2} ds\\
	&\quad + C\int_0^t(1+t-s)^{-\frac54}\| h_4(s)\|_{L^1\cap L^2} + (1+t-s)^{-1}\| \partial_x^2 h_4(s)\|_{L^2} ds \\
	&\le C(1+t)^{-\frac34}\delta_0 \\
	&\quad + CQ^2(t)\int_0^t(1+t-s)^{-\frac54}(1+s)^{-\frac32} + (1+t-s)^{-1}(1+s)^{-\frac94} ds\\
	&\quad + CQ^2(t)\int_0^t(1+t-s)^{-\frac54}(1+s)^{-2} + (1+t-s)^{-1}(1+s)^{ -\frac{13}6} ds\\
	&\le C(1+t)^{-\frac34}\delta_0 + C(1+t)^{-1}Q^2(t),\\
	\| \partial_x E_r\|_{L^2}  &\le \| \partial_x(G_e^{21}\ast u_{r0})\|_{L^2} +\| \partial_x (G_e^{22}\ast E_{r0})\|_{L^2} +\| \partial_x (G_e^{23}\ast B_{r0})\|_{L^2}\\
	&\quad +\int_0^t \| \partial_x (G_e^{21}(t-s)\ast h_3(s))\|_{L^2} + \| \partial_x (G_e^{22}(t-s)\ast h_4(s))\|_{L^2}ds\\
	&\le C(1+t)^{-\frac54}( \| (u_{r0},E_{r0},B_{r0})\|_{L^1\cap L^2}+ \| \partial_x^4(u_{r0},E_{r0},B_{r0})\|_{L^2})\\
	&\quad +C\int_0^t(1+t-s)^{-\frac74}\| h_3(s)\|_{L^1\cap L^2} ds + C\int_0^{t/2}(1+t-s)^{-\frac32} \| \partial_x^3h_3(s)\|_{L^2}ds\\
	&\quad  + C\int_{t/2}^t(1+t-s)^{-1}\| \partial_x^2 h_3(s)\|_{L^2} ds\\
	&\quad + C\int_0^t(1+t-s)^{-\frac74}\| h_4(s)\|_{L^1\cap L^2}ds +C\int_0^{t/2} (1+t-s)^{-\frac32}\| \partial_x^4 h_4(s)\|_{L^2} ds \\
	&\quad + C\int_{t/2}^t (1+t-s)^{-\frac12}\| \partial_x^2 h_4(s)\|_{L^2}ds\\
	&\le C(1+t)^{-\frac54}\delta_0 \\
	&\quad + C\int_0^t(1+t-s)^{-\frac74}(1+s)^{-\frac74}Q^2(t)ds +C\int_0^{t/2}(1+t-s)^{-\frac32}(1+s)^{-\frac{13}{12}}Q^2(t)ds\\
	&\quad  +C\int_{t/2}^t(1+t-s)^{-1}(1+s)^{-\frac{5}3}Q^2(t) ds\\
	&\quad + C\int_0^t(1+t-s)^{-\frac74}(1+s)^{-2}Q^2(t) ds + C\int_0^{t/2}(1+t-s)^{-\frac32}(1+s)^{-1}Q^2(t) ds\\
	&\quad +C\int_{t/2}^t (1+t-s)^{-\frac12} (1+s)^{ -\frac{13}{6}}Q^2(t)ds\\
	& \le C(1+t)^{-\frac54}\delta_0 + C(1+t)^{-\frac32}\ln{(2+t)} Q^2(t).
\emas

For $B_r(t)$, we have
\bmas
	\| B_r\|_{L^2} &\le \|  G_e^{31}\ast u_{r0}\|_{L^2} +\|  G_e^{32}\ast  E_{r0}\|_{L^2}+\|  G_e^{33}\ast  B_{r0}\|_{L^2}\\
	&\quad +\int_0^t \|  G_e^{31}(t-s)\ast  h_3(s)\|_{L^2} + \| G_e^{32}(t-s)\ast h_4(s)\|_{L^2}ds\\
	&\le C(1+t)^{-\frac14}( \| (u_{r0},E_{r0},B_{r0})\|_{L^1\cap L^2} +\| \partial_x(u_{r0},E_{r0},B_{r0})\|_{L^2})\\
	&\quad + C\int_0^t(1+t-s)^{-\frac34}\| h_3(s)\|_{L^1\cap L^2} + (1+t-s)^{-\frac12}\| h_3(s)\|_{L^2} ds\\
	&\quad + C\int_0^t(1+t-s)^{-\frac34}\| h_4(s)\|_{L^1\cap L^2}+ (1+t-s)^{-\frac12}\| \partial_x h_4(s)\|_{L^2} ds \\
	&\le C(1+t)^{-\frac14}\delta_0\\
	&\quad + CQ^2(t) \int_0^t(1+t-s)^{-\frac34}(1+s)^{-\frac32} + (1+t-s)^{-\frac12}(1+s)^{-\frac74} ds\\
	&\quad + CQ^2(t) \int_0^t(1+t-s)^{-\frac34}(1+s)^{-2} + (1+t-s)^{-\frac12}(1+s)^{-\frac{11}4} ds\\
	&\le C(1+t)^{-\frac14}\delta_0 + C(1+t)^{-\frac12}Q^2(t),\\
	\| \partial_x B_r\|_{L^2} &\le \| \partial_x (G_e^{31}\ast u_{r0})\|_{L^2} +\|  \partial_x (G_e^{32}\ast  E_{r0})\|_{L^2}+\|  \partial_x (G_e^{33}\ast B_{r0})\|_{L^2}\\
	&\quad +\int_0^t \|  \partial_x (G_e^{31}(t-s)\ast  h_3(s))\|_{L^2} + \| \partial_x( {G_e^{32}(t-s)\ast h_4(s)})\|_{L^2}ds\\
	&\le C(1+t)^{-\frac34}( \| (u_{r0},E_{r0},B_{r0})\|_{L^1\cap L^2} +\| \partial_x^3(u_{r0},E_{r0},B_{r0})\|_{L^2})\\
	&\quad + C\int_0^t (1+t-s)^{-\frac54}\| h_3(s)\|_{L^1\cap L^2} +(1+t-s)^{-1}\| \partial_x^2 h_3(s)\|_{L^2}ds\\
	&\quad + C\int_0^t (1+t-s)^{-\frac54}\| h_4(s)\|_{L^1\cap L^2} +(1+t-s)^{-1}\| \partial_x^3 h_4(s)\|_{L^2}ds \\
	&\le C(1+t)^{-\frac34}\delta_0\\
&\quad +CQ^2(t) \int_0^t (1+t-s)^{-\frac54}(1+s)^{-\frac32}+ (1+t-s)^{-1}(1+s)^{-\frac{5}3} ds\\
	&\quad + CQ^2(t) \int_0^t (1+t-s)^{-\frac54}(1+s)^{-2} + (1+t-s)^{-1}(1+s)^{-\frac{19}{12}} ds\\
	&\le  C(1+t)^{-\frac34}\delta_0 +C(1+t)^{-1}Q^2(t).
\emas
And from the energy estimate, we obtain
$$
	\| \partial_x^2(\rho,u_1,u_r,B_r)\|_{H^2}  \le C \| U(0)\|_{H^4} \le C\delta_0.
$$
Therefore, we conclude that
$$
	Q(t) \le C\delta_0 + CQ^2(t),
$$
from which the claim \eqref{claim} holds for $\delta_0$ being small enough and hence we complete the proof.
\end{proof}

\begin{thm}
Assume that $U_r(0)=(u_{r0},E_{r0},B_{r0})^{T}\in H^{N}\cap L^1$ with $N\ge 4$, and $\| U_r(0)\|_{H^N}\le \delta_0$ with $\delta_0 > 0$ being sufficiently small. There exists a constant $d_0>0$ such that $\inf\limits_{|k|\le\epsilon} |\hat{B}_{r0}(k)|\ge d_0$, then the solution $U_r=(u_r,E_r,B_r)^{T}$ of Euler-Maxwell system \eqref{re-em_1}--\eqref{re-initial} satisfies for $\alpha=0,1$,
\bgr
	C_1d_0(1+t)^{-\frac{3}{4}-\frac{\alpha}{2}} \le \| \partial_x^\alpha(u_r,E_r)(t)\|_{L^2} \le C_2\delta_0(1+t)^{-\frac{3}{4}-\frac{\alpha}{2}},\\
	C_1d_0(1+t)^{-\frac{1}{4}-\frac{\alpha}{2}} \le \| \partial_x^\alpha B_r(t)\|_{L^2} \le C_2\delta_0(1+t)^{-\frac{1}{4}-\frac{\alpha}{2}},
\egr
where $C_2\ge C_1>0$ are two constants and $t>0$ is sufficiently large.
\end{thm}

\begin{proof}
By Lemma \ref{Ge-lowerbound} and Theorem \ref{nonlinear-decay}, we only need to show the lower bounds of the time decay rates. For $\alpha =0,1$, from \eqref{duhamel_r}, one has
\bmas
	\| \partial_x^\alpha u_r\|_{L^2} &\ge \| \partial_x^\alpha(G_e^{13} \ast B_{r0}) \|_{L^2} - \| \partial_x^\alpha(G_e^{11} \ast u_{r0}) \|_{L^2} - \| \partial_x^\alpha(G_e^{12} \ast E_{r0}) \|_{L^2}  \\
	&\quad - \int_0^t\| \partial_x^\alpha[G_e^{11}(t-s) \ast h_3(s) + G_e^{12}(t-s) \ast h_4(s)] \|_{L^2} ds \\
	&\ge C_1d_0(1+t)^{-\frac{3}{4}-\frac{\alpha}{2}} - C_2\delta_0(1+t)^{-\frac{5}{4}-\frac{\alpha}{2}} -C_2\delta_0^2(1+t)^{-\frac34-\frac{\alpha}{2}}\\
	&\ge  C_3d_0(1+t)^{-\frac34-\frac{\alpha}{2}},\\
	\| \partial_x^\alpha E_r\|_{L^2} &\ge \| \partial_x^\alpha(G_e^{23} \ast B_{r0}) \|_{L^2} - \| \partial_x^\alpha(G_e^{21} \ast u_{r0}) \|_{L^2} - \| \partial_x^\alpha(G_e^{22} \ast E_{r0}) \|_{L^2} \\
	&\quad- \int_0^t\| \partial_x^\alpha[G_e^{21}(t-s) \ast h_3(s) + G_e^{22}(t-s) \ast h_4(s)] \|_{L^2} ds \\
	&\ge C_1d_0(1+t)^{-\frac{3}{4}-\frac{\alpha}{2}} - C_2\delta_0(1+t)^{-\frac{5}{4}-\frac{\alpha}{2}} -  C_2\delta_0(1+t)^{-\frac54-\frac{\alpha}{4}} -C_2\delta_0^2(1+t)^{-\frac34-\frac{\alpha}{2}}\\
	&\ge  C_3d_0(1+t)^{-\frac34-\frac{\alpha}{2}}.
\emas
Here, from \eqref{Ge-22} in Lemma \ref{Ge-convolution}, for $\alpha=1$, we use the following estimate
\bmas
	\|\partial_x(G_e^{22}\ast E_{r0})\|_{L^2} \le C(1+t)^{-\frac74}\| E_{r0}\|_{L^1\cap L^2} + C(1+t)^{-\frac32}\| \partial_x^4 E_{r0}\|_{L^2} \le C_2\delta_0(1+t)^{-\frac32}.
\emas
And, we also have
\bmas
	\| \partial_x^\alpha B_r\|_{L^2} &\ge \|\partial_x^\alpha( G_e^{33} \ast B_{r0}) \|_{L^2} - \| \partial_x^\alpha(G_e^{31} \ast u_{r0}) \|_{L^2} - \| \partial_x^\alpha(G_e^{32} \ast E_{r0}) \|_{L^2} \\
	&\quad - \int_0^t\| \partial_x^\alpha[G_e^{31}(t-s) \ast h_3(s) + G_e^{32}(t-s) \ast h_4(s)] \|_{L^2} ds \\
	&\ge C_1d_0(1+t)^{-\frac{1}{4}-\frac{\alpha}{2}} - C_2\delta_0(1+t)^{-\frac{3}{4}-\frac{\alpha}{2}} -C_2\delta_0^2(1+t)^{-\frac14-\frac{\alpha}{2}}\\
	&\ge  C_3d_0(1+t)^{-\frac14-\frac{\alpha}{2}}.
\emas
The above estimates are holds for sufficiently large $t>0$ and small $\delta_0>0$.
\end{proof}
							
\medskip
\noindent {\bf Acknowledgements:}  This work is supported by the National Natural Science Foundation of China  grants No 12171104, the special foundation for Guangxi Ba Gui Scholars and Innovation Project of Guangxi Graduate Education YCBZ2024004.

\end{document}